\documentclass[12pt]{amsart}

\usepackage{tensor}     

\usepackage{esint}

\usepackage{tikz}
\usetikzlibrary{calc, decorations.pathreplacing, arrows.meta, patterns, cd}

\usepackage[colorlinks=true,urlcolor=blue, citecolor=red,linkcolor=blue,linktocpage,pdfpagelabels, bookmarksnumbered,bookmarksopen]{hyperref}
\usepackage[hyperpageref]{backref}
\usepackage{cleveref}
\usepackage{xcolor}
\usepackage{amsthm} 
\usepackage{latexsym,amsmath,amssymb}
\usepackage{accents}
\usepackage[colorinlistoftodos,prependcaption,textsize=tiny]{todonotes}
\usepackage{a4wide}
\usepackage{soul}
\usepackage{mathtools} 
\usepackage{xparse} 
\usepackage{enumitem}		
\usepackage{listings}
\usepackage{textcomp}

\usepackage{calc}

\usepackage{accents}

\definecolor{indigo}{rgb}{0.29, 0.0, 0.51}
\definecolor{p1}{gray}{0.4}
\definecolor{p2}{gray}{0.6}
\definecolor{p3}{gray}{0.98}
\definecolor{p4}{gray}{0.8}
\definecolor{p5}{gray}{0.9}

\newcommand{\Kasia}[1]{\todo[color=magenta!30]{{\bf Kasia:} #1}}
\newcommand{\ikasia}[1]{\todo[inline, color=magenta!30]{{\bf Kasia:} #1}}

\newcommand*{\red}{\textcolor{red}}

plus 6pt minus 12pt
\def\eps{\varepsilon}

\newcommand{\m}{\S^3 }

\def\B{{B}}

\def\n{{\mathcal N}}

\def\N{{\mathbb N}}
\def\n{{\S^3 }}
\def\n{{\mathcal N}}

\def\S{{\mathbb S}}

\newcommand{\degh}{\deg_{H}}

\newtheorem{theorem}{Theorem}
\newtheorem{lemma}[theorem]{Lemma}
\newtheorem{corollary}[theorem]{Corollary}
\newtheorem{proposition}[theorem]{Proposition}

\newtheorem{remark}[theorem]{Remark}
\newtheorem{definition}[theorem]{Definition}
\newtheorem{example}[theorem]{Example}

\def\curl{{\rm curl\,}}

\def\supp{{\rm supp\,}}

\newcommand{\dext}{\mathrm{d}}
\newcommand{\dif}{\,\mathrm{d}}
\newcommand{\dx}{\,\dif x}
\newcommand{\dy}{\,\dif y}
\newcommand{\dz}{\,\dif z}

\newcommand{\dd}{\dif}

\newcommand{\e}{\textnormal{e}}

\newcommand{\R}{\mathbb{R}}
\newcommand{\bB}{\mathbb{B}}

\newcommand{\Z}{\mathbb{Z}}
\newcommand{\C}{\mathbb{C}}
\newcommand{\CP}{\mathbb{CP}}
\newcommand{\Lip}{\mathrm{Lip}}

\newcommand{\brac}[1]{\left (#1 \right )}
\newcommand{\abs}[1]{\left |#1 \right |}

\newcommand{\norm}[1]{\left\|{#1}\right\|}

\newcommand{\la}{\mathopen{}\mathclose\bgroup\left\langle}
\newcommand{\ra}{\aftergroup\egroup\right\rangle}

\DeclareMathOperator{\Hom}{Hom}

\DeclareMathOperator{\Jac}{Jac}

\newcommand{\barint}{
\rule[.036in]{.12in}{.009in}\kern-.16in \displaystyle\int }

\newcommand{\barcal}{\mbox{$ \rule[.036in]{.11in}{.007in}\kern-.128in\int $}}

\def\mvint_#1{\mathchoice
          {\mathop{\vrule width 6pt height 3 pt depth -2.5pt
                  \kern -8pt \intop}\nolimits_{\kern -3pt #1}}%
          {\mathop{\vrule width 5pt height 3 pt depth -2.6pt
                  \kern -6pt \intop}\nolimits_{#1}}%
          {\mathop{\vrule width 5pt height 3 pt depth -2.6pt
                  \kern -6pt \intop}\nolimits_{#1}}%
          {\mathop{\vrule width 5pt height 3 pt depth -2.6pt
                  \kern -6pt \intop}\nolimits_{#1}}}

\numberwithin{theorem}{section} \numberwithin{equation}{section}

\newcommand{\aleq}{\precsim}

\def\XXint#1#2#3{{\setbox0=\hbox{$#1{#2#3}{\int}$}
     \vcenter{\hbox{$#2#3$}}\kern-.5\wd0}}

\let\latexchi\chi
\makeatletter
\renewcommand\chi{\@ifnextchar_\sub@chi\latexchi}
\newcommand{\sub@chi}[2]{
  \@ifnextchar^{\subsup@chi{#2}}{\latexchi^{}_{#2}}%
}
\newcommand{\subsup@chi}[3]{
  \latexchi_{#1}^{#3}%
}
\makeatother

\usepackage{mathtools} 

\title[Minimizing $n$-harmonic maps in homotopy classes]{Existence of minimizing $n$-harmonic maps in homotopy classes by critical exponents}

\author{Antoine Detaille}
\address[Antoine Detaille]{
ETH Z\"urich\\
Departement Mathematik\\
R\"amistrasse 101\\
8092 Z\"urich \\
Switzerland}
\email{antoine.detaille@math.ethz.ch }

\author{Adam Grzela}
\address[Adam Grzela]{Institute of Mathematics, %
University of Warsaw,
Banacha 2,
02-097 Warszawa, Poland}
\email{a.grzela@uw.edu.pl}

\author{Katarzyna Mazowiecka}
\address[Katarzyna Mazowiecka]{
Institute of Mathematics, %
University of Warsaw,
Banacha 2,
02-097 Warszawa, Poland}
\email{k.mazowiecka@mimuw.edu.pl}

\begin{document}
\subjclass[2010]{58E20, 35B65, 35J60, 35S05}
\begin{abstract}
We investigate the existence of energy minimizing $n$-harmonic maps from an $n$-dimensional sphere to a compact manifold inside a prescribed homotopy class.
Our results are of two kinds.

First, inspired by the work by T.~Rivière (1998) on minimizing 3-harmonic maps from $\mathbb{S}^{3}$ to $\mathbb{S}^{2}$ with given Hopf degree and by the work of R.~Hardt and T.~Rivière (2008) on critical and optimal Sobolev exponents for rational homotopy classes, we exhibit a general principle that guarantees the existence of infinitely many homotopy classes with a minimizing map inside.
We apply this principle, each time proving the optimality of the relevant Sobolev exponents, to mappings from $\mathbb{S}^{4n-1}$ to $\mathbb{S}^{2n}$, where the relevant homotopical quantity is the Hopf invariant, and to mappings from $\mathbb{S}^{2n+1}$ to $\mathbb{CP}^{n}$, where homotopy classes are generated by the circle bundle on the projective space.
These results hold for first-order as well as for fractional Sobolev mappings.

Second, elaborating on the recent preprint of Wang~G. and Zhang~M. (2026), we derive sharp lower energy bounds for mappings with given Hopf invariant. 
We apply these bounds to deduce a sufficient energetic condition for the existence of minimizing maps in a given homotopy class. 
Finally, we exploit this condition, by explicit construction of mappings satisfying it, to prove the existence of energy minimizers from $\mathbb{S}^{3}$ to $\mathbb{S}^{2}$ of Hopf degree $d \in \{\pm2,\pm3,\dotsc,\pm6\}$ and from $\mathbb{S}^{7}$ to $\mathbb{S}^{4}$ of Hopf degree $\pm1$.
 \end{abstract}
\maketitle
\sloppy

\maketitle
\tableofcontents
\sloppy

\section{Introduction}\label{s:introduction}
Let $\n\subseteq \R^L$ be a closed Riemannian manifold, assumed to be embedded without loss of generality by the Nash theorem. For $1\leq p \leq\infty$, let us denote by $W^{1,p}(\S^n, \n)$ the standard Sobolev space of mappings from $\S^n$ into $\n$, endowed with the metric inherited from $W^{1,p}(\S^n,\R^L)$. We consider the energy
\begin{equation}\label{eq:13-energy}
 E_{1,p}(u) = \int_{\S^n} |\nabla u|^p\text{.}
\end{equation}
Analogously, for $s\in(0,1)$, we denote by $W^{s,p}(\S^n,\n)$ the Sobolev--Slobodeckij space with the fractional energy
\begin{equation}\label{eq:s3-energy}
E_{s,p}(u) = 
\int_{\mathbb{S}^n} \int_{\mathbb{S}^n} \frac{|u(x)-u(y)|^{p}}{|x-y|^{n+sp}} \dif x\dif y\text{.}
\end{equation}

The problem of finding minimizers\footnote{These are called minimizing
$W^{s,\frac ns}$-harmonic maps or, when $s=1$, minimizing
$n$-harmonic maps.} of $E_{s,\frac{n}{s}}$ for $s\in(0,1]$ in homotopy classes of $\pi_n(\n)$ is an old and well established one, interesting from the perspective of both analysis and geometry.
In the following, and without claim of exhaustivity, we draw a brief history of this problem and recall some known results that are related to the main goals of this paper.

Perhaps the most simple case is \( \n = \S^n \), where homotopy classes are classified by the degree of Brouwer.
For maps from $\S^2$ to itself, every homotopy class contains a
minimizing harmonic map: these are given by rational maps,
their complex conjugates, and constant maps; see \cite{Lemaire}.
For $n\geq 3$ however, the $n$-energy of maps from $\S^n$ to itself
admits a minimizer only in the classes of degree $0$ and $\pm1$.
This distinction reflects the rigidity of conformal mappings in
higher dimensions, which follows from Liouville's theorem; see
\cite[p.2]{Xu-Yang93} and \cite[Introduction]{Riviere98}.
Using the fact that the covering \( \pi \colon \S^{n} \to \mathbb{RP}^{n} \) by the antipodal map is a local isometry that preserves homotopy classes, those results adapt mutatis mutandis replacing \( \n  = \S^{n} \) with \( \n = \mathbb{RP}^{n} \); see Remark~\ref{rem:RPn} for a similar adaptation of our main results.

For general target manifolds, the seminal work of Sacks and
Uhlenbeck \cite{SacksUhlenbeck} established a variational approach
to the existence of harmonic spheres and the analysis of the
loss of compactness caused by bubbling. Subsequent developments
extended this theory to higher-dimensional and nonlocal energies. The following theorem summarizes the resulting existence theory:
\begin{theorem}[{\cite{SacksUhlenbeck, NakauchiTakakuwa, KawaiNakauchiTakeuchi, Riviere98, Duzaar-Kuwert, Sucks}}]\label{th:SUn-dim}
 Assume that $\pi_1(\mathcal N) =\{0\}$ and $s\in(0,1]$. 
 There exists a generating set of homotopy classes in $C^0(\S^n,\n)$ in which there is a minimizing $W^{s,\frac ns}$-harmonic map.
\end{theorem}
This theorem guarantees the existence of a generating set of homotopy classes admitting minimizers, but does not identify which classes belong to this set.

An interesting example is about the \( 3 \)-energy of mappings from \( \S^{3} \) to \( \S^{2} \), where homotopy classes are classified by Hopf's degree.
In \cite{Riviere98}, Rivi\`{e}re proved the following result:
\begin{theorem}[{\cite[Theorem I.1]{Riviere98}}]\label{th:Riviere}
There exist infinitely many homotopy classes of $\pi_{3}(\S^2)$ having minimizing $3$-harmonic maps, i.e., a map minimizing the $E_{1,3}$ energy in its homotopy class.
\end{theorem}

The proof of \Cref{th:Riviere} is based on asymptotic relations between the homotopy type of maps $u\colon \S^3\to\S^2$ and their energy. To be more precise, the homotopy group $\pi_3(\S^2)$ is isomorphic to $\Z$ by what is known as the \emph{Hopf invariant} (or the \emph{Hopf degree}). More generally, to a given map $f\colon \S^{4n-1}\to\S^{2n}$, the Hopf invariant associates a topological invariant given by an integer.
In the case $n=1$, this integer completely determines the homotopy class --- for \( n \geq 2 \), the group \( \pi_{4n-1}(\S^{2n}) \) may also have a torsion part, which is not detected by the Hopf invariant. 
For a smooth map, a classical construction due to Whitehead \cite{Whitehead_1947} gives an integral formula for the Hopf invariant. 

To state Whitehead's integral formula for the Hopf invariant, we first consider the standard volume form \(\omega_{\mathbb{S}^{2n}}\) on \(\mathbb{S}^{2n}\). The pullback \(f^* \omega_{\mathbb{S}^{2n}}\) is a closed \(2n\)-form on \(\mathbb{S}^{4n-1}\), and since \( H^{2n}_{\textnormal{dR}}(\S^{4n-1}) = \{0\} \), this form is exact.
Thus, there exists a \((2n-1)\)-form \(\eta \in \Omega^{2n-1}(\mathbb{S}^{4n-1})\) such that \(\dext\eta = f^* \omega_{\mathbb{S}^{2n}}\). The Hopf degree of \(f\) is then given by
\begin{equation}\label{eq:Hopfdegree}
\degh(f) = \frac{1}{|S^{2n}|^2}\int_{\mathbb{S}^{4n-1}} \eta \wedge \dext\eta\text{.}
\end{equation}
This definition is independent of the choice of \(\eta\) --- and also of the choice of the volume form \( \omega_{\S^{2n}} \) on \( \S^{2n} \) --- and the resulting quantity \(\degh\) is invariant under homotopy. For further properties of the Hopf invariant, we refer to~\cite{BottTu}.

The Hopf invariant admits an estimate in terms of the critical Sobolev semi-norm; see~\cite[Introduction]{Riviere98}:
\begin{equation}\label{eq:hopf-estimate}
\left| \degh(f) \right| \lesssim \left( \int_{\mathbb{S}^{4n-1}} |\nabla f|^{4n-1} \right)^{\frac{4n}{4n-1}}\text{.}
\end{equation}
This establishes a relation between the topological and analytical properties of a map, controlling its topological complexity by its energy. It is important to note that the exponent \(\frac{4n}{4n-1}>1\), and one might naturally ask whether this exponent could be improved. However, in the case \(n = 1\), Rivi\`{e}re showed in~\cite{Riviere98} that the exponent \(\frac{4}{3}\) in~\eqref{eq:hopf-estimate} is optimal. More precisely, he proved that
\begin{equation}\label{eq:optimalexponent}
\frac{\log \#_{1,3} |d|}{\log |d|} \to \frac{3}{4} \quad \text{as} \quad |d| \to \infty\text{,}
\end{equation}
where 
for every $s\in(0,1]$ and $\alpha\in\pi_n(\n)$, we define
\begin{equation}\label{eq:infenergysn}
    \#_{s,\frac ns}\alpha
    \coloneqq
    \inf\left\{
        E_{s,\frac ns}(u)
        \colon
        u\in W^{s,\frac ns}(\S^n,\n),\ [u]=\alpha
    \right\}\text{.}
\end{equation}
Here, \( [u] \) denotes the homotopy class of \( u \).
As \( s\frac{n}{s} = n \) coincides with the dimension of the domain \( \S^{n} \), homotopy classes of \( W^{s,\frac{n}{s}}(\S^{n}, \n) \) maps are well-defined; see, e.g., \cite{BrezisNirenberg1995}.

When prescribing an integer-valued homotopy invariant
$\deg\colon\pi_n(\n)\to\Z$, we use the notation
\begin{equation}\label{eq:infimumoffractionalenergyhopf}
    \#_{s,\frac ns}d
    \coloneqq
    \inf\left\{
        E_{s,\frac ns}(u)
        \colon
        u\in W^{s,\frac ns}(\S^n,\n),\ \deg([u])=d
    \right\}
\end{equation}
for every admissible value $d\in \deg(\pi_n(\n))$.
Here $\deg$ is the Brouwer degree or the Hopf invariant,
according to the setting.
Later on, we will consider more general integer-valued homotopy invariants, see Section~\ref{subsect:RealHomotopy}.

Equation \eqref{eq:sub-linear-local} establishes the asymptotic behavior of the infimal energy as the Hopf invariant grows to infinity. Crucially, the energy scales sub-linearly with the invariant, as \eqref{eq:optimalexponent} implies that
\begin{equation}\label{eq:sub-linear-local}
\#_{1,3} d \lesssim |d|^{\frac{3}{4}}\text{.}
\end{equation}

The other side of the relationship between the homotopy type and the energy is the homotopy-energy decomposition principle known as the \emph{energy identity}. This identity takes various forms and we state a general one in the spirit of the seminal work of Sacks and Uhlenbeck \cite[Lemma 5.4]{SacksUhlenbeck}:

\begin{theorem}[Energy identity, {\cite[Proposition III.1]{Riviere98}}, {\cite[Theorem 6]{Duzaar-Kuwert}}, {\cite[Theorem 1]{KawaiNakauchiTakeuchi}}, {\cite[Theorem 3.2]{MS-stability}}]\label{th:energy-id-general}
Let $\n$ be a compact simply connected Riemannian manifold and $n \geq 2$. For any $u \in W^{s,\frac ns}\left(\S^n, \n\right)$, let $[u]$ denote its homotopy class in $\pi_n(\n)$. Then, there exist finitely many $W^{s,\frac ns}$-harmonic maps $u^{(1)}, \dotsc, u^{(k)} \in W^{s,\frac ns}\left(\S^n, \n\right)$ which satisfy the following conditions:
\begin{enumerate}[label=(\arabic*),ref=(\arabic*)]
    \item\label{it:homotopy-decomposition}
    $[u]=[u^{(1)}]+\dotsb+[u^{(k)}]$;
    \item\label{it:energy-decomposition}
    $\displaystyle
    \inf_{w\in[u]} E_{s,\frac{n}{s}}(w)
    =E_{s,\frac{n}{s}}(u^{(1)})
    +\dotsb+E_{s,\frac{n}{s}}(u^{(k)})$;
    \item\label{it:component-minimizers}
    $u^{(j)}$ is a minimizer of $E_{s,\frac{n}{s}}$
    in $[u^{(j)}]$, for $j=1,\dotsc,k$.
\end{enumerate}
\end{theorem}

In order to prove \Cref{th:Riviere}, having \eqref{eq:optimalexponent} and \Cref{th:energy-id-general} at hand, one can argue by contradiction. Suppose only finitely many homotopy classes admit minimizers.
 Then, letting $d_1$ denote the largest absolute value among these classes and $d_0$ the class with the smallest energy, the energy identity \Cref{th:energy-id-general}\ref{it:energy-decomposition} implies that
\begin{equation}
 \forall d\in\Z\text{,} \quad \#_{1,3} d \ge \frac{\#_{1,3} d_0}{|d_1|} |d|\text{,}
\end{equation}
which contradicts \eqref{eq:optimalexponent} as $|d|\to\infty$. We note that this argument does not require any knowledge of the constant involved in the inequality \eqref{eq:sub-linear-local}. 

The arguments sketched above seem fairly general and in fact rely only on three basic ingredients: 
\begin{enumerate}
\item Existence of an energy identity;
\item Existence of a quantitative homotopy invariant that can be estimated in terms of the energy;
\item Optimality of the above mentioned estimate.
\end{enumerate}
In our main result, we manage to circumvent the last ingredient, which is usually the most problematic. 
Indeed, the full picture for optimality of homotopy-energy estimates is still far from being understood.

For quantitative homotopy invariants, we can turn to \emph{real homotopy theory}. In particular, we can rely on a framework developed by Novikov \cite{Novikov_1984,Novikov_1985,Novikov_1988} and later expanded by Hardt and Rivi\`{e}re \cite{HardtRiviere08} to derive a large class of homotopy invariants, that admit quantitative estimates in terms of the critical Sobolev and fractional Sobolev energy. 
We present a brief overview and introduction to these notions in Section~\ref{s:prelim}. 

The problem lies, of course, in proving, that the critical estimates are indeed optimal. In general, establishing optimality of exponents in estimates like \eqref{eq:hopf-estimate} appears to be a delicate question; see~\cite[Section 2.5.2 and Proposition 2.15]{HardtRiviere08} for related discussions. 

The purpose of this paper is to generalize \Cref{th:Riviere}. In particular, we prove the following:

\begin{theorem}\label{th:main}
Let $s\in (0,1]$, and let $\n$ be a closed, simply connected Riemannian manifold. 
Assume that $\deg\in \Hom(\pi_n(\n), \R)$ is a nontrivial real-valued homotopy invariant such that there exists a sequence $d_k\to\infty$ satisfying
\begin{equation}\label{eq:conditionforinftyhomotopyclasses}
\lim_{k\to \infty} \frac{1}{d_k} \inf_{\deg(u) = d_k} E_{s,\frac{n}{s}}(u) = 0\text{.}
\end{equation}
Then, there exist infinitely many homotopy classes in $\pi_n(\n)$ admitting a minimizer of $E_{s,\frac{n}{s}}$. 
\end{theorem}

This is closely related to the notions of critical and optimal exponents for rational homotopy invariants (originally defined by Hardt and Rivi\`{e}re \cite{HardtRiviere08}), which we describe in \Cref{s:prelim} and later in this section. As corollaries, we recover \Cref{th:Riviere} and its fractional counterpart, as well as versions for higher Hopf invariants:
\begin{corollary}\label{th:higher-hopf}
Let $s\in(0,1]$, $n\in \N$. There exist infinitely many homotopy classes of $\pi_{4n-1}(\S^{2n})$ containing a map minimizing the energy functional $E_{s,\frac{4n-1}{s}}$. 
\end{corollary}
We remark that the same results are possible for the real projective space:
\begin{remark}\label{rem:RPn}
Let $s\in(0,1]$, $n\in\N$.
There exist infinitely many homotopy classes of
$\pi_{4n-1}(\mathbb{RP}^{2n})$ containing a map minimizing the energy functional
$E_{s,\frac{4n-1}{s}}$.
\end{remark}

Additionally, we obtain a similar result for another natural generalization of the Hopf map, involving complex projective spaces:
\begin{corollary}\label{cor:cpn}
Let $s\in(0,1]$. There exist infinitely many homotopy classes of $\pi_{2n+1}(\CP^{n})$ containing a map minimizing the energy functional $E_{s,\frac{2n+1}{s}}$. 
\end{corollary}

Let us mention, to give some actual context to Corollary~\ref{cor:cpn}, that maps to complex projective spaces have received quite a lot of interest recently, with works notably focused on the study of their Morse index or on minimizing tangent maps; see \cite{ChenWangZhang2026,GutweinLanglais2026,Surkov2026,Lara2026,GianoccaGuerraZiener2026} and the references therein.

A common feature of all the above results we presented, namely the general principle from Theorem~\ref{th:main} and its Corollaries~\ref{th:higher-hopf} and~\ref{cor:cpn} and Remark~\ref{rem:RPn} in special cases, is that they are non-constructive existence results.
More specifically, they guarantee the existence of infinitely many homotopy classes with a minimizing map, but without providing any explicit example of a homotopy class with a minimizing map inside.
For instance, in the case of the \( 3 \)-energy of mappings from \( \S^{3} \) to \( \S^{2} \), it is natural to expect that the class of mappings of Hopf degree equal to \( 1 \) contains minimizing maps, and that one of them is the Hopf map itself, as was conjectured by Rivière \cite{Riviere98}.

Recently, Wang and Zhang \cite{Wang-Zhang} established a sharp
lower bound for the $3$-energy of maps from \( \S^{3} \) to \( \S^{2} \) in terms of the Hopf invariant.
In particular, for maps of Hopf invariant equal to \( 1 \), this bound is attained precisely by the Hopf map, or any orientation-preserving conformal transformation of it.
This yields existence of minimizing degree \( 1 \) maps, and provides a complete characterization of them \cite[Theorem~1.4]{Wang-Zhang}, answering Rivière's question. 
Combining this bound with an existence criterion and explicit
comparison maps, we obtain the following result concerning maps of degree \( \pm2, \pm3, \dotsc, \pm6 \).
\begin{theorem}\label{pr:hopf234}
    There exist minimizers of $E_{1,3}$ in homotopy classes of $\pi_3(\S^2)$ corresponding to the Hopf degree $d$ for every $|d|\le 6$.
\end{theorem}
We also examine how the minimum energy depends on the Hopf degree.
For maps from $\S^2$ to $\S^2$, the classical identity
\[
    \#_{1,2} d = |d|\#_{1,2} 1
\]
shows that the minimum Dirichlet energy depends linearly on the
absolute value of the degree; see, for instance, \cite{Lemaire}
or \cite[Theorem~4.1]{Helein}. Similarly, for maps from $\S^1$
to $\S^1$ and $\frac14<s<1$, we have
\[
    \#_{s,\frac1s} d = |d|\#_{s,\frac1s} 1\text{;}
\]
see \cite[Theorem~1.3]{MMS26}.

For maps from $\S^3$ to $\S^2$, the sublinear upper bound
\eqref{eq:sub-linear-local} rules out such a linear law.
More precisely, we prove that the minimum energy is strictly
below the linear bound and strictly above the sharp lower bound
whenever $d\ge2$.

\begin{proposition}\label{pr:f(d)behavior}
Let
\[
    f(d)\coloneqq \frac{\#_{1,3} d}{\#_{1,3} 1}\text{.}
\]
Then, for every integer $d\ge2$,
\begin{equation}\label{eq:f(d)behavior}
    d^{3/4}<f(d)<d\text{.}
\end{equation}
\end{proposition}
Finally, for maps from $\S^7$ to $\S^4$, we prove that the
minimum energy among maps of Hopf degree one is attained.
\begin{theorem}\label{pr:existenceindeg1n=2}
There exists a map \(u\in W^{1,7}(\S^7,\S^4)\) with
\(\deg_H(u)=1\) such that
\[
    E_{1,7}(u)=\#_{1,7}1\text{.}
\]
\end{theorem}

\begin{remark}
We expect the quaternionic Hopf map to attain this infimum,
but we do not establish its minimality here.
\end{remark}

\subsection{Outline of the paper}
In \Cref{s:prelim}, we give an brief overview of the required theory of real homotopy, and recall the relevant results of Novikov and Hardt and Rivi\`{e}re. In \Cref{s:main}, we give a proof of the main result \Cref{th:main}. The next two sections generalize \eqref{eq:optimalexponent}, treating the local and nonlocal cases separately. In \Cref{s:hopfs}, we extend Rivi\`ere's construction
\cite{Riviere98} for maps from $\S^3$ to $\S^2$ to maps from
$\S^{4n-1}$ to $\S^{2n}$ in higher dimensions proving
\Cref{th:higher-hopf}. We conclude the section with a sketch of proof of \Cref{rem:RPn}. In \Cref{s:cpn}, we give a construction
for maps from $\S^{2n+1}$ to $\CP^n$ that yields the
corresponding result, \Cref{cor:cpn}. Finally, in \Cref{s:higher-degree}, we prove the existence of minimizers in specific homotopy classes and establish bounds for the minimum energy $E_{1,3}$ in terms of the Hopf degree.

\subsection*{Notation} We write $A\lesssim B$ to indicate that there exists a constant $C$ such that $A \le CB$, where $C$ is independent of any essential parameter. 

For an $n$-dimensional domain $\Omega$ we write $E_{s,p}(u,\Omega)$, defined as in \eqref{eq:13-energy} and \eqref{eq:s3-energy}, with integration over $\Omega$ when $s=1$ and over $\Omega\times \Omega$ when $s\in(0,1)$. In particular, we use $E_{s,p}(u,\S^n)$ and $E_{s,p}(u)$ interchangeably.  

\subsection*{Acknowledgment} The project is co-financed by
\begin{itemize}
	\item (AD) the Swiss National Science Foundation, project SNF 200020\_219429;
	\item (AG, KM) the Polish National Agency for Academic Exchange within Polish Returns Programme -
	BPN/PPO/2021/1/00019/U/00001;
	\item (AG, KM) the National Science Centre, Poland grant No.\ 2023/51/D/ST1/02907;
	\item (AG) the National Science Centre, Poland grant No.\ 2025/57/N/ST1/04534.
\end{itemize}

\subsection*{AI disclosure}
Most of the results of this paper were obtained before July 2025 and no artificial intelligence was used. After the preprint \cite{Wang-Zhang} was published on arXiv the authors hoped that these results may imply existence of minimizers in every homotopy class.  \Cref{thm:higherdegreemaps} was obtained by ChatGPT Pro, the prompt asked to look for maps $u_d$ of Hopf degree $d$ satisfying \eqref{eq:energyboundforexistencen} for $n=1$ and $d\ge 2$ of the form $u_d=H\circ w_d$, where $w_d\colon \S^3\to\S^3$ and $H$ is the Hopf map for as many $d$ as possible. This provided the first draft of \Cref{a:aiexplicit}, where the proof of \Cref{thm:higherdegreemaps} is. After that, tens of subsequent prompts were employed to force the current presentation. The authors checked
and adjusted \Cref{a:aiexplicit}, notably to improve the readability and pedagogy of the presentation, they take full responsibility for it. 

\section{Preliminaries}\label{s:prelim}

\subsection{On maps from \texorpdfstring{$\S^{4n-1}$}{S4n-1} to \texorpdfstring{$\S^{2n}$}{S2n}}

As an important part of our work is concerned with mappings \( \S^{4n-1} \to \S^{2n} \) and the associated Hopf invariant, for the convenience of the reader and to fix the notation, we start with some reminders.

As we already explained in the introduction, there is a nontrivial group homomorphism \( \deg_H \colon \pi_{4n-1}(\S^{2n}) \to \Z \) given by
\[
	\deg_H(f)
	=
	\frac{1}{|\S^{2n}|^2}\int_{\S^{4n-1}} \eta \wedge \dext \eta\text{,}
\]
where \( \eta \in \Omega^{2n-1}(\S^{4n-1}) \) is any form such that \( \dext\eta = f^{\ast}\omega_{\S^{2n}} \); see \cite{Whitehead_1947}.
For \( n = 1 \), this homomorphism actually realizes the isomorphism \( \pi_{3}(\S^{2}) \cong \Z \).
On the other hand, if \( n \geq 2 \), then \( \deg_H \) need not be a monomorphism nor an epimorphism.
Indeed, for \( n \geq 2 \), \( \pi_{4n-1}(\S^{2n}) \) is isomorphic to the direct product of \( \Z \) with a finite group, which need not be trivial.
Being a \( \Z \)-valued homomorphism, \( \deg_H \) is insensitive to this torsion part, hence thwarting injectivity.
Moreover, as we recall below in Theorem~\ref{th:hopfdeg1}, there need not be a map of Hopf degree equal to \( 1 \), hence thwarting surjectivity.
Finally, for the curiosity of the reader, we recall that, according to a theorem of Serre \cite[§ V Proposition~5]{Serre_1951}, the only homotopy groups of spheres of infinite order --- and hence, on which the approach we feature here can be applied --- are \( \pi_{n}(\S^{n}) \) and \( \pi_{4n-1}(\S^{2n}) \).

We now become more specific, and recall two fundamental statements regarding the existence of maps of Hopf degree equal to \( 1 \) and \( 2 \).

\begin{theorem}[{\cite[Satz VII]{Hopf1935}, {\cite[Theorem 1.1.1]{Adams60}}}]\label{th:hopfdeg1}
	For $n=1,2,4$, there exists a map $f\colon \S^{4n-1}\to\S^{2n}$ such that $\deg_H(f)=1$. 
	For $n\notin \{1,2,4\}$, there is no map $f\colon \S^{4n-1}\to\S^{2n}$ of Hopf degree \(1\).
\end{theorem}

\begin{theorem}[see, e.g., {\cite[p.428]{Hatcher}}]\label{th:hopfdeg2}
	For all $n \geq 1$, there is a map $f\colon \S^{4n-1}\to\S^{2n}$ of Hopf degree \(2\).
\end{theorem}

The combination of Theorems~\ref{th:hopfdeg1} and~\ref{th:hopfdeg2} above implies that \( \deg_H \) is an epimorphism from \( \pi_{4n-1}(\S^{2n}) \) into \( \Z \) for \( n = 1,2,4 \), and into \( 2\Z \) for \( n \notin \{1,2,4\} \).

The map in \Cref{th:hopfdeg2} can be relatively easily constructed via the \emph{Whitehead product} of identity maps. Explicitly, such a map can be given by the formula
\begin{equation}\label{eq:hopfdeg2}
	\Phi(x,y) = 
	\begin{cases}
		\brac{4\sqrt{2}\frac{1-2|x|^2}{(1+2|x|^2)^2} x,\ 1 - 2\brac{\frac{1-2|x|^2}{1+2|x|^2}}^2} &\text{for } |x|<|y|\text{,} \\
		\brac{4\sqrt{2}\frac{1-2|y|^2}{(1+2|y|^2)^2} y,\ 1 - 2\brac{\frac{1-2|y|^2}{1+2|y|^2}}^2} &\text{for } |y|<|x|\text{,} \\
		x_N  &\text{for } |x| = |y|\text{,}
	\end{cases}
\end{equation}
where \[\mathbb{S}^{4n-1} = \left\{ (x,y)\in \R^{2n}\oplus\R^{2n}\colon |x|^2+|y|^2=1 \right\}\] and the spheres in $\mathbb{S}^{2n}\vee \mathbb{S}^{2n}$ are wedged together at $x_N$, the north pole. 

If instead one is willing to identify \( \S^{4n-1} \) with \( \partial Q^{4n} \), the boundary of a \( 4n \)-dimensional cube, then there is a much simpler formula for the Whitehead product; see, e.g., \cite[Section~4.1]{DetailleVanSchaftingen2024}.
However, this is not suitable for our purposes, as the bi-Lipschitz transformation to move from \( \S^{4n-1} \) to \( \partial Q^{4n} \) would destroy the sharp energy estimates that we need here.

Let us quickly explain where the formula \eqref{eq:hopfdeg2} comes from. The map $\Phi$ can be obtained on the two components $\S_{x}$ and $\S_{y}$ defined as
\[
\S_{x} = \left\{ (x,y)\in \S^{4n-1}\colon |x|<|y| \right\} \quad 
\text{and}
\quad
\S_{y} = \left\{ (x,y)\in \S^{4n-1}\colon |x|>|y| \right\}\text{.}
\]
We define the rescaled projections $P_x\colon\S_{x}\to\R^{2n}$ and $P_y\colon\S_{y}\to\R^{2n}$ as
\begin{equation}\label{eq:projections}
	P_{x}(x,y) = \sqrt{2}x \quad
	\text{and}
	\quad 
	P_{y}(x,y) = \sqrt{2}y\text{.}
\end{equation}
Then, in terms of those projections, the map $\Phi$ can be obtained simply via the composition
\[
\Phi = g \circ \Pi^{-1} \circ P_i\text{,}
\]
where $P_i$ is the appropriate projection from \Cref{eq:projections}, $\Pi^{-1}$ is the inverse stereographic projection and \( g \colon \mathbb{S}^m \to \mathbb{S}^m \) is the smooth map
\[
g(x) = (-2x_1x_{m+1}, \dotsc, -2x_mx_{m+1}, 1 - 2x_{m+1}^2)\text{,}
\]
which maps the southern hemisphere onto \( \mathbb{S}^m \) and satisfies \( g(x) = x_N \) on the equator.

We observe that $\Phi$ is a continuous, piecewise smooth map with uniformly bounded gradient. Moreover, since it is a piecewise composition of projections with two diffeomorphisms, the appropriate Jacobians $\Phi^*\omega_{\S^{2n}}$ are bounded from below, away from zero. A classical computation shows that $\deg_H(\Phi) = 2$ \cite[XI, 2.5]{Whitehead78book}. In particular, the preimage of a regular point $q$, $\Phi^{-1}(\{q\})$, is a disjoint sum of linked spheres $\mathbb{S}^{2n-1}_{r}\cup \mathbb{S}^{2n-1}_{r}$ of radius $r=r(q)\in(0,1)$ depending on $q$.

Regarding Theorem~\ref{th:hopfdeg1}, classical explicit examples of maps of Hopf invariant equal to \( 1 \) are also available, and are provided by the Hopf map.

When \( n = 1 \), we work with the unit spheres
\[
\S^3=\{(z_1,z_2)\in\C^2\colon |z_1|^2+|z_2|^2=1\}
\qquad
\text{and}
\qquad
\S^2=\{(\zeta,x)\in\C\times\R\colon |\zeta|^2+x^2=1\}\text{,}
\]
with their induced round metrics and boundary orientations, identifying \( \C \) with \( \R^2 \). 
The real
coordinates on $\C^2$ are ordered as
$(\operatorname{Re}z_1,\operatorname{Im}z_1,
\operatorname{Re}z_2,\operatorname{Im}z_2)$.
The Hopf map \( H \colon \S^{3} \to \S^{2} \) is then defined as
\begin{equation}\label{eq:Hopf}
	H(a,b)=(2a\overline b,\,|a|^2-|b|^2)\text{,}\qquad (a,b)\in\S^3\text{.}
\end{equation}
It satisfies the properties
\[
\deg_{H}(H) = 1
\qquad
\text{and}
\qquad
\abs{\nabla H}^{2} = 8\text{.}
\]
Hence,
\begin{equation}
\label{eq:energyhopf}
E_{1,3}(H)
=
8^{3/2}\abs{\S^{3}}
=
16\sqrt{2}2\pi^{2}
=
32\sqrt{2}\pi^{2}\text{.}
\end{equation}

Similarly, when \( n = 2 \), we work with the unit spheres
\[
\S^7 = \{(a,b)\in\mathbb{H}\colon |a|^2+|b|^2=1\} \qquad
\text{and}
\quad 
\S^4=\{(x,t)\in\mathbb{H}\times \R\colon |x|^2 + t^2=1\}\text{,}
\]
with \( \mathbb{H} \) denoting the quaternions.
The quaternionic Hopf map \( H_{q} \colon \S^{7} \to \S^{4} \) is then defined as
\begin{equation}\label{eq:quaternionicHopf}
	H_q(a,b) = (2ab,|a|^2-|b|^2)\text{,} \quad (a,b) \in \S^7\text{;}
\end{equation}
see \cite[Example~4.46]{Hatcher}.
It satisfies the properties
\[
\deg_{H}(H_{q}) = 1
\qquad
\text{and}
\qquad
\abs{\nabla H_{q}}^{2} = 16\text{.}
\]
Hence,
\begin{equation}
	\label{eq:energyqhopf}
	E_{1,7}(H_{q})
	=
	16^{7/2}\abs{\S^{7}}
	=
	\frac{4^{7}\pi^{4}}{3}\text{.}
\end{equation}

There is also an octonionic Hopf map \( \S^{15} \to \S^{8} \), but we omit its description here, as it shall play no role in this work.

To conclude this section, we recall two useful formulas describing how the Hopf invariant of a map behaves under composition, both on the left and on the right.

\begin{proposition}
\label{prop:HopfDegLeftComp}
	Let \( f \colon \S^{4n-1} \to \S^{2n} \) and \( g \colon \S^{2n} \to \S^{2n} \) be continuous maps.
	Then, 
	\[
		\deg_H(g \circ f) = \deg(g)^{2}\deg_H(f)\text{.}
	\]
\end{proposition}

\begin{proposition}
	\label{prop:HopfDegRightComp}
	Let \( f \colon \S^{4n-1} \to \S^{4n-1} \) and \( g \colon \S^{4n-1} \to \S^{2n} \) be continuous maps.
	Then, 
	\[
	\deg_H(g \circ f) = \deg(f)\deg_H(g)\text{.}
	\]
\end{proposition}

The proofs being quite short, we present them for the convenience of the reader.

\begin{proof}[Proof of Proposition~\ref{prop:HopfDegLeftComp}]
	By Brouwer's formula for the degree,
	\[
		\deg(g)
		=
		\int_{\S^{2n}} g^{\ast}\omega_{\S^{2n}}\text{,}
	\]
	so that \( \frac{g^{\ast}\omega_{\S^{2n}}}{\deg(g)} \) is a volume form on \( \S^{2n} \).
	Choosing \( \eta \in \Omega^{2n-1}(\S^{4n-1}) \) such that
	\[
		\dext\eta
		=
		f^{\ast}\biggl(\frac{g^{\ast}\omega_{\S^{2n}}}{\deg(g)}\biggr)\text{,}
	\]
	we have
	\[
		\dext(\deg(g)\eta)
		=
		f^{\ast}(g^{\ast}\omega_{\S^{2n}})
		=
		(g \circ f)^{\ast}\omega_{\S^{2n}}\text{.}
	\]
	Whitehead's formula then implies that
	\[
	\begin{split}
		\deg_H(g \circ f)
		&=
		\frac{1}{|\S^{2n}|^2}\int_{\S^{4n-1}} \deg(g)\eta \wedge \dext(\deg(g)\eta) \\
		&=
		\deg(g)^{2} \frac{1}{|\S^{2n}|^2}\int_{\S^{4n-1}}\eta \wedge \dext\eta
		=
		\deg(g)^{2}\deg_H(f)\text{.}
	\end{split}
	\]
	This concludes the proof.
\end{proof}

\begin{proof}[Proof of Proposition~\ref{prop:HopfDegRightComp}]
	Let \( \eta \in \Omega^{2n-1}(\S^{4n-1}) \) be such that \( \dext\eta = g^{\ast}\omega_{\S^{2n}} \).
	We have
	\[
		\dext(f^{\ast}\eta)
		=
		f^{\ast}(\dext\eta)
		=
		f^{\ast}(g^{\ast}\omega_{\S^{2n}})
		=
		(g \circ f)^{\ast}\omega_{\S^{2n}}\text{.}
	\]
	Hence, Whitehead's formula implies that
	\[
		\deg_H(g \circ f)
		=
		\frac{1}{|\S^{2n}|^2} \int_{\S^{4n-1}} f^{\ast}\eta \wedge \dext(f^{\ast}\eta)
		=
		\frac{1}{|\S^{2n}|^2}\int_{\S^{4n-1}}f^{\ast}(\eta \wedge \dext\eta)
		=
		\deg(f)\deg_H(g)\text{,}
	\]
	where, in the last equality, we have used Brouwer's formula for the degree and the fact that, once normalised by \( \frac{1}{|\S^{2n}|^2}\deg_H(g) \), the form \( \eta \wedge \dext\eta \) is a volume form on \( \S^{4n-1} \) by Whitehead's formula.
\end{proof}

\subsection{Overview on real homotopy}\label{subsect:RealHomotopy}

We next make some reminders on the theory of real homotopy, which will be our source of invariants to produce minimizing maps in homotopy classes.
Define the $n$-th \emph{real homotopy group} of $\n$ to be
\begin{equation}
\pi^{\R}_{n}(\n) = \Hom(\pi_n(\n),\R)\text{.}
\end{equation}
In the usual context, one often works with \emph{rational homotopy}, but for our purposes it is more convenient to use $\R$, which does not change any relevant feature in our case. 
We will sometimes refer to elements of $\pi_n^\R(\n)$ as \emph{generalized degrees}. 
One can view any generalized degree $\deg\in\pi_n^\R(\n)$ as acting on continuous maps $u\colon \S^n\to \n$, by considering
\begin{equation}
\deg(u) \coloneqq \deg([u]),
\end{equation}
where $[u]$ denotes the element of $\pi_n(\n)$ represented by $u$. As such, this action can be extended whenever homotopy classes are well defined, for instance for maps in the Sobolev space $W^{s,p}(\S^n, \n)$ for $sp\geq n$; see \cite{BrezisNirenberg1995}.

Rational (or indeed real) homotopy types of compact simply connected manifolds is a relatively well-understood topic. 
As we already mentioned, the classification of rational homotopy of spheres is due to Serre \cite{Serre_1951}, who proved, using spectral sequences, that the only homotopy groups containing non-torsion parts for spheres are $\pi_n(\S^n)$ and $\pi_{4n-1}(\S^{2n})$. 
Classical results due to Sullivan \cite{Sullivan_1975, Sullivan_1977} give a full classification in terms of \emph{differential graded algebras} --- in particular, the \emph{de Rham complex} of $\n$ or, more precisely, its \emph{minimal model}. This allows for computations and explicit formulas for a large class of homotopy invariants. 

\begin{definition}
	A \emph{differential graded algebra} (or DGA) is a graded vector space $A = \bigoplus_{i\in \N} A_i$ with a graded-commutative product, that is,
	\begin{equation}
	a \wedge b = (-1)^{|a| |b|} b\wedge a\text{,}
	\end{equation}
	and an anti-derivation $\dext$ of degree 1, satisfying
	\begin{equation}
	\dext^2 = 0 \quad
	\text{and}
	\quad 
	\dext(a\wedge b) = \dext a \wedge b + (-1)^{|a|}a\wedge \dext b\text{,}
	\end{equation}
	where by $|a|$ we denote the degree (grading) of $a$. 
\end{definition}
Note that, by this definition, a DGA can be considered as a cochain complex of vector spaces, and therefore has a well defined notion of cohomology. A generic and very important example of a DGA is of course the algebra of smooth differential forms on a manifold, $\Omega^*(\n) = \bigoplus_{k=0}^{n} \Omega^k(\n)$. Its cohomology is then simply the de Rham cohomology of $\n$.
A DGA is \emph{free} whenever it has no other relation than the skew-symmetry of the product.

\begin{definition}
	A differential graded algebra $B$ is said to be \emph{minimal} if it is free as a graded-commutative algebra on generators of degrees $\geq 2$ and satisfies $\dext B \subset B^+ \wedge B^+$, where $B^+ = \bigoplus_{i>0} B_i$.
	
	A \emph{minimal model} for a differential graded algebra $A$, denoted $\mathcal{M}_A$, is a minimal differential graded algebra together with a morphism of DGAs $\Psi_A\colon \mathcal{M}_A\to A$ (called the \emph{geometric realization}) such that $\Psi_A$ induces an isomorphism on cohomology. 
\end{definition}

To illustrate and explain the above definition, let us provide an example, connected to the special cases of degrees we considered so far; see, e.g., \cite[Examples~19.1 and~19.2]{BottTu}.

\begin{example}\label{ex:even-sphere}
For $\n = \S^{2n}$, the algebra of smooth differential forms $\Omega^*(\S^{2n})$ has only two non-vanishing cohomology groups: $H^0_{\textnormal{dR}}(\S^{2n})\simeq\R$ and $H^{2n}_{\textnormal{dR}}(\S^{2n})\simeq\R$. Therefore, the minimal model can be computed as the free algebra on two generators of degrees $2n$ and $4n-1$:
\begin{equation}
\mathcal{M}_{\S^{2n}} = \bigwedge[x_{2n}, y_{4n-1}]\text{,}
\end{equation}
with the relation $\dext y_{4n-1} = x_{2n}\wedge x_{2n}$. The geometric realization can be taken as mapping $x_{2n}\mapsto\omega_{\S^{2n}}$, a volume form on the sphere, and $y_{4n-1}\mapsto 0$. 

Let us give some extra intuition on how to come to this construction.
A naive idea would be to simply consider the free algebra on the single generator \( x_{2n} \) of grading \( 2n \), that is, \( \bigwedge(x_{2n}) \), so that we can map \( x_{2n} \) onto \( \omega_{\S^{2n}} \).
However, since the grading of \( x_{2n} \) is even, its successive products do not vanish, as this is not imposed by the skew-symmetry relation.
Actually, the subspace of elements of grading \( k2n \) is the vector line generated by \( x_{2n} \wedge \dotsb \wedge x_{2n} \), taking a product of \( k \) copies of \( x_{2n} \).
Since there is no element of grading \( k2n-1 \) and \( k2n+1 \), \( x_{2n} \wedge \dotsb \wedge x_{2n} \) is closed but not exact, and actually \( H^{k2n}(\bigwedge(x_{2n})) \simeq \R \).
Hence, we would not obtain an isomorphism in cohomology.

To circumvent this issue, we add an extra generator of grading \( 4n-1 \), denoted \( y_{4n-1} \) and mapped to \( 0 \) in \( H^{\ast}_{\textnormal{dR}}(\S^{2n}) \), and the relation \( \dext y_{4n-1} = x_{2n} \wedge x_{2n} \), so that \( x_{2n} \wedge x_{2n} \) becomes exact.
Let us note that this element does not add any contribution in cohomology.
Indeed, as \( \dext y_{4n-1} = x_{2n} \wedge x_{2n} \neq 0 \), \( y_{4n-1} \) is not closed.
Moreover, as \( y_{4n-1} \) is of odd grading, all its products with itself vanish by the skew-symmetry relation, so that they also do not contribute to the cohomology.

By the same argument, we observe that, for spheres of odd dimension,
\[
	\mathcal{M}_{\S^{2n+1}} = \bigwedge(x_{2n+1})\text{,}
\]
and the geometric realization is obtained by mapping \( x_{2n+1} \) to \( \omega_{\S^{2n+1}} \).
This does not only induce an isomorphism on cohomology, but even provides an isomorphism of DGA.
Indeed, as \( x_{2n+1} \) has odd grading, the skew-symmetry relation forces all its products with itself to vanish, and therefore \( \mathcal{M}_{\S^{2n+1}} \) defined as above just consists of a vector line of grading \( 2n+1 \), all the spaces of other gradings being trivial.
This is exactly the cohomology ring \( H^{\ast}_{\textnormal{dR}}(\S^{2n+1}) \).
\end{example}

We provide another example, involving projective spaces, motivated by Corollary~\ref{cor:cpn}; see, e.g., \cite[Example~19.3]{BottTu}.

\begin{example}
\label{ex:cpn_minimal_model}
	We claim that
	\[
		\mathcal{M}_{\CP^{n}} = \bigwedge[x_{2},y_{2n+1}]
	\]
	with the relations \( \dext x_{2} = 0 \) and \( \dext y_{2n+1} = x_{2}^{n+1} \).
	Indeed, the cohomology ring of \( \CP^{n} \) is given by \( H^{\ast}(\CP^{n}) = \R[x_{2}]/(x_{2}^{n+1}) \) where \( x_{2} \) has grading \( 2 \).
	One therefore wishes to map \( x_{2} \) onto the Kähler form \( \omega \in \Omega^{2}(\CP^{n}) \), but one needs to kill \( x_{2}^{n+1} \) in cohomology.
	This is precisely the reason for introducing \( y_{2n+1} \), mapped to \( 0 \) in \( \Omega^{\ast}(\CP^{n}) \), such that \( \dext y_{2n+1} = x_{2}^{n+1} \).
	In particular, \( y_{2n+1} \) is not closed, and hence does not contribute to the cohomology.
	Therefore, this procedure indeed provides a minimal model for \( \CP^{n} \).
\end{example}

For more involved examples, we refer to \cite[Section~2.3]{HardtRiviere08}.

The key insight of Sullivan was to relate the rational homotopy type of a simply connected manifold to its algebra of differential forms via minimal models. More precisely, the following theorems holds.

\begin{theorem}[{Sullivan, \cite{Sullivan_1977}}]\label{thm:sul1}
For any compact simply connected manifold $\n$, the exterior algebra of differential forms on $\n$, $A=\Omega^*(\n)$, admits a minimal model $\mathcal{M}_{\n}$. 
\end{theorem}

\begin{theorem}[{Sullivan, \cite{Sullivan_1977}}]\label{thm:sul2}
The space $\Hom(\pi_n(\n),\R)$ is isomorphic to the $\R$-vector space spanned by the generators of degree $n$ in $\mathcal{M}_{\n}$.
\end{theorem}

\subsection{d-extensions and Novikov integral formulas}
A technical construction of \emph{d-extensions} of minimal models due to Novikov \cite{Novikov_1984,Novikov_1985,Novikov_1988} allows one give concrete expressions in terms of integrals of differential forms. 

\begin{definition}
The \emph{\(k\)-th \(\operatorname{d}\)-extension} of a minimal model $\mathcal{M}_{\n}$, denoted $C_k(\mathcal{M}_{\n})$, is a DGA such that all the lower cohomology groups vanish, $H^p(C_k(\mathcal{M}_{\n})) \simeq \{0\}$ for $p=1,\dotsc,k$, and $H^{k+1}(C_k(\mathcal{M}_{\n})) \simeq \Hom(\pi_{k+1}(\n),\R)$. 
\end{definition}

Let us note that a realization of the real homotopy group with a cohomology group allows for the explicit use of the de Rham isomorphism, and hence gives explicit formulas in terms of integration on the sphere $\S^n$. 

\begin{example}
In Example \ref{ex:even-sphere}, the minimal model
\begin{equation}
\mathcal{M}_{\S^{2n}} = \bigwedge[x_{2n}, y_{4n-1}]
\end{equation}
with the relation \( \dext y_{4n-1} = x_{2n} \wedge x_{2n} \)
identifies two non-zero real homotopy groups for $\S^{2n}$ according to Theorem \ref{thm:sul2}, namely $\pi_{2n}^\R(\S^{2n})$ and $\pi_{4n-1}^\R(\S^{2n})$. The appropriate \(\dext\)-extensions can be computed as 
\begin{equation}
C_{2n-1}(\mathcal{M}_{\S^{2n}}) = \mathcal{M}_{\S^{2n}} \quad \text{and} \quad 
C_{4n-2}(\mathcal{M}_{\S^{2n}}) = \mathcal{M}_{\S^{2n}}[\gamma_{2n-1}]\text{,}
\end{equation}
where $\gamma_{2n-1}$ is a free generator satisfying $\dext\gamma_{2n-1} = x_{2n}$. The non-vanishing cohomology groups $H^{2n}(C_{2n-1}(\mathcal{M}_{\S^{2n}}))$ and $H^{4n-1}(C_{4n-2}(\mathcal{M}_{\S^{2n}}))$ are generated by $x_{2n}$ and $y_{4n-1} - x_{2n}\wedge\gamma_{2n-1}$ respectively. 
The corresponding homotopy integrals can be computed as
\begin{equation}
x_{2n}([u]) = \int_{\S^{2n}} u^*\omega_{\S^{2n}} \quad \text{and} \quad 
y_{4n-1}([u]) = \int_{\S^{4n-1}} u^*\omega_{\S^{2n}}\wedge d^{-1}(u^*\omega_{\S^{2n}}),
\end{equation} 
recovering the standard formulas for the Brouwer degree and the Hopf degree; see \cite{Whitehead_1947}.

Indeed, for \( C_{2n-1}(\mathcal{M}_{\S^{2n}}) \), it is clear that \( H^p(\mathcal{M}_{\S^{2n}}) \simeq \{0\} \) for \( p = 1, \dotsc, 2n-1 \) and that \( H^{2n}(\mathcal{M}_{\S^{2n}}) \simeq \R \simeq \Hom(\pi_{2n}(\S^{2n}),\R) \).
Let us note that the last identity, relating \( H^{2n}(\mathcal{M}_{\S^{2n}}) \) and \( \Hom(\pi_{2n}(\S^{2n}),\R) \), can be obtained \emph{without} knowing a priori \( \pi_{2n}(\S^{2n}) \), by using Theorem~\ref{thm:sul2}.
From this and the fact that \( x_{2n} \) is mapped onto \( \omega_{\S^{2n}} \), the announced formulas for \( C_{2n-1}(\mathcal{M}_{\S^{2n}}) \) and \( x_{2n}([u]) \) follow directly.

For \( C_{4n-2}(\mathcal{M}_{\S^{2n}}) \), we just need to kill \( H^{2n}(\mathcal{M}_{\S^{2n}}) \).
To do this, we introduce an extra free generator \( \gamma_{2n-1} \) such that \( \dext\gamma_{2n-1} = x_{2n} \).
Being of odd grading, \( \gamma_{2n-1} \) satisfies \( \gamma_{2n-1} \wedge \gamma_{2n-1} = 0 \), so that no contribution is added at grading \( 4n-2 \).
The vector subspace of grading \( 4n-1 \) of the resulting DGA \( C_{4n-2}(\mathcal{M}_{\S^{2n}}) \) obtained by this process is generated by \( y_{4n-1} \) and \( x_{2n} \wedge \gamma_{2n-1} \), which is such that
\[
	\dext(x_{2n} \wedge \gamma_{2n-1})
	=
	\dext x_{2n}\wedge \gamma_{2n-1} + x_{2n} \wedge \dext\gamma_{2n-1}
	=
	x_{2n} \wedge x_{2n}
	=
	\dext y_{4n-1}\text{.}
\]
Therefore, \( y_{4n-1} - x_{2n} \wedge \gamma_{2n-1} \) is closed, but cannot be exact as there is no nontrivial element of grading \( 4n-2 \).
The space \( H^{4n-1}(C_{4n-2}(\mathcal{M}_{\S^{2n}})) \simeq \R \simeq \Hom(\pi_{4n-1}(\S^{2n}),\R) \) is generated by \( y_{4n-1} - x_{2n} \wedge \gamma_{2n-1} \), showing that we indeed constructed the \( (4n-2) \)-th \( \dext \)-extension of \( \mathcal{M}_{\S^{2n}} \).
Here again, the last equality can be obtained form Theorem~\ref{thm:sul2}, without knowing a priori \( \pi_{4n-1}(\S^{2n}) \).
The formula for \( y_{4n-1}([u]) \) follows from this and the fact that \( y_{4n-1} \) is mapped to \( 0 \), that \( x_{2n} \) is mapped to \( \omega_{\S^{2n}} \) and that \( \dext\gamma_{2n-1} = x_{2n} \).
\end{example}

We also pursue Example~\ref{ex:cpn_minimal_model} about complex projective spaces.

\begin{example}
\label{ex:cpn_degre_integral}
	Clearly, the \( 2 \)-nd \( \dext \)-extension of \( \mathcal{M}_{\CP^{n}} \) is \( \mathcal{M}_{\CP^{n}} \) itself, and since \( x_{2} \) is mapped onto the Kähler form \( \omega \in \Omega^{2}(\CP^{n}) \), we have for the natural degree on \( \pi_{2}(\CP^{n}) \simeq \Z \) that
	\[
		x_{2}([u])
		=
		\int_{\S^{2}} u^{\ast}\omega\text{.}
	\]
	
	Now, we kill \( H^{2}(\mathcal{M}_{\CP^{n}}) \), which is generated by \( x_{2} \), by introducing an element \( \gamma_{1} \) of grading \( 1 \) such that \( \dext \gamma_{1} = x_{2} \).
	Doing so, we also have \( \dext(\gamma_{1} \wedge x_{2}^{k-1}) = \dext\gamma_{1} \wedge x_{2}^{k-1} = x_{2}^{k} \) for every \( k \geq 1 \), so that all the products of the \( x_{2} \) are killed as well.
	This yields a DGA \( \mathcal{M}_{\CP^{n}}[\gamma_{1}] \), whose space of grading \( 2k \) is generated by \( x_{2}^{k} \) and whose space of grading \( 2k+1 \) is generated by \( \gamma_{1} \wedge x_{2}^{k} \) as long as \( k < n \), and moreover, its cohomology vanish in all degree less than \( 2n+1 \) by construction.
	For the cohomology in degree \( 2n+1 \), we have two generators in the corresponding vector space of \( \mathcal{M}_{\CP^{n}}[\gamma_{1}] \), namely \( \gamma_{1} \wedge x_{2}^{n} \), which is not closed as we have just seen, and \( y_{2n+1} \), which is also not closed by construction.
	But, we have
	\[
		\dext(y_{2n+1} - \gamma_{1} \wedge x_{2}^{n})
		=
		x_{2}^{n+1} - x_{2}^{n+1}
		=
		0\text{,}
	\]
	so that \( y_{2n+1} - \gamma_{1} \wedge x_{2}^{n} \) is closed, and generates \( H^{2n+1}(\mathcal{M}_{\CP^{n}}[\gamma_{1}]) \).
	Since the vector space of grading \( 2n+1 \) is one-dimensional, we deduce that \( H^{2n+1}(\mathcal{M}_{\CP^{n}}[\gamma_{1}]) \simeq \R \simeq \Hom(\pi_{2n+1}(\CP^{n}),\R) \) by virtue of Theorem~\ref{thm:sul2}, and hence
	\[
		C_{2n}(\mathcal{M}_{\CP^{n}})
		=
		\mathcal{M}_{\CP^{n}}[\gamma_{1}]\text{.}
	\]
	As \( x_{2} \) is mapped onto the Kähler form \( \omega \) and \( y_{2n+1} \) onto \( 0 \), and as \( \gamma_{1} \) satisfies \( \dext\gamma_{1} = x_{2} \), the corresponding degree is computed by the integral formula
	\[
		y_{2n+1}([u])
		=
		\int_{\S^{2n+1}} \dext^{-1}(u^{\ast}\omega) \wedge u^{\ast}\omega\text{.}
	\]
\end{example}

For more involved examples, we refer to \cite[Section~2.3]{HardtRiviere08}.

\subsection{Critical and optimal exponents}
With the goal of making the construction of the aforementioned homotopy integrals more explicit and more adapted for estimates, Hardt and Rivi\`{e}re \cite{HardtRiviere08} pursued the work of Sullivan and Novikov, associating to any homotopy invariant a so-called \emph{tree-graph} and providing analytic estimates in terms of Sobolev energies $E_{1,n}$. 
The tree-graph is a combinatorial way of representing the invariant.
More specifically, the invariant is encoded by a formal linear combination of trees, with vertices representing pullbacks of differential forms and edges representing the ``primitive form'' operation $\dext^{-1} \coloneqq \dext^*\Delta^{-1}$, defined on the sphere $\S^n$, where $\dext^*$ is the Hodge co-differential and $\Delta$ denotes the Hodge--Laplace operator. 
We do not recall the detailed construction of the tree-graph representation, but rather refer the reader to \cite[Sections 2.2. and 2.3.]{HardtRiviere08}. 
We just mention that this representation need not be unique: The same invariant may very well be represented by different tree graphs, which yield different differential forms giving the same result once integrated on \( \S^{n} \).
In particular, it is possible to ensure that all the forms labelling the vertices of the trees are \emph{closed}, and as we will see, the minimal such tree, that is, the one with the minimal amount of edges, will be of particular interest.
It is also possible, albeit more complicated, to extend this analysis to fractional Sobolev--Slobodeckij spaces $W^{s,\frac{n}{s}}(\S^n, \n)$; see \cite{ParkSchikorra}.

Hardt and Rivière also introduced the notion of \emph{critical exponent} for the estimate of the invariant, which is computed in terms of the corresponding tree-graph. More precisely, we have the following definition:
\begin{definition}[{\cite[Section 2.5.1.]{HardtRiviere08}}]\label{def:critical-exp}
Let $\deg \in \pi_n^\R(\n)$ be an element of the $n$-th real homotopy group. The \emph{critical exponent} for $\deg$ is a rational number of the form 
\begin{equation}
p = \frac{n+q_{\deg} }{n}\text{,}
\end{equation}
where $q_{\deg}$ is the number of edges of the optimal tree-graph representation. More precisely, if $\mathfrak{K}$ is the set of all (formal combinations of) trees of closed differential forms $K=\sum_{i\in I} \lambda_i K_i$ representing the element $\deg$, we have
\begin{equation}
q_{\deg} \coloneqq \min_{K\in\mathfrak{K}} \max_{i\in I} |E(K_i)|\text{,}
\end{equation}
where $|E(K_i)|$ denotes the number of edges in the graph $K_i$.
\end{definition}

In light of the discussion from \Cref{s:introduction}, it is tempting to ask whether such critical exponents are in some sense asymptotically optimal.
This motivates the following definition:
\begin{definition}[{\cite[Section 2.5.2.]{HardtRiviere08}}]\label{def:optimal-exp}
Let $\n, \deg\in\pi_{n}^{\R}(\n)$ be as above. The critical exponent $p$ for $\deg$ is called \emph{optimal} for $E_{1,n}$ whenever there exists a sequence $d_k \xrightarrow{k\to\infty} \infty$
such that
\begin{equation}
\lim_{k\to\infty} \frac{\log{d_k}}{\log{\#_{1,n}d_k}} = p\text{,}
\end{equation}
where, in accordance with \eqref{eq:infimumoffractionalenergyhopf}, we denote
\[
	\#_{1,n}d_k = \inf_{\deg(u)=d_k} E_{1,n}(u)\text{.}
\]
\end{definition}
The above definition can be trivially modified to include the functional $E_{s,\frac{n}{s}}$. Let us note, as already explained in \Cref{s:introduction}, that the inequality 
\[
	\#_{s,n} d \lesssim \#_{1,n} d
\] 
always holds. 

Some results on optimality of exponents (albeit cast in a different language) have been given by Brezis, Coron and Lieb \cite{Brezis_1986} and Rivi\`{e}re \cite{Riviere98}. More precisely, Brezis, Coron and Lieb proved, by a direct construction of maps, that for the Brouwer degree $\deg$ of maps in $\pi_{n}(\S^n)$ for $n\geq 3$, we have
\begin{equation}
\#_{1,n} d = |\S^n| n^{n/2} |d|\text{.}
\end{equation}
In the language of critical exponents, this shows that for the Brouwer degree, the exponent \(1\) is indeed optimal. On the other hand, Rivi\`{e}re showed in \cite[Lemma III.1]{Riviere98} that for the Hopf degree $\deg_H$ of maps in $\pi_{3}(\S^{2})$, we have
\begin{equation}
\#_{1,3} d \lesssim |d|^{\frac{3}{4}}\text{.}
\end{equation}
This shows that the critical exponent $\frac{4}{3}$ (cf.\ \Cref{eq:hopf-estimate}) is indeed optimal for the Hopf degree.

Beyond that, Hardt and Rivi\`{e}re \cite[Proposition~2.15]{HardtRiviere08} provide an example family of manifolds --- $\S^2\times\S^2$, $\CP^2$, $\overline{\CP}^2$ and connected sums thereof --- for which the critical exponents are always optimal. A comprehensive analysis of this topic is still missing in the literature. 

\subsection{Tools for estimates in the fractional Sobolev spaces}
We now move from homotopical invariants to Sobolev mappings.
As our work is concerned also with fractional Sobolev spaces, we will need to cope with extra difficulties arising from the \emph{nonlocal} character of the Gagliardo seminorm.
We therefore present a toolbox instrumental for our purposes.

The following technical lemmas will be used to estimate the fractional Sobolev energy of maps defined on unions of sets or modified locally. The first lemma allows us to localize the $W^{s,p}$-energy of a map with a ``buffer zone'':
\begin{lemma}[{\cite[Lemma 2.2]{MonteilVanSchaftigen}}]\label{le:gluing}
Let $s\in (0,1]$ and $1\leq p< \infty$, and let $\n$ be a connected Riemannian manifold. There exists a constant $C= C(s,p,m)>0$ such that, for every $\eta\in (0,1)$, every open set $A\subset \R^m$, every measurable function $u\colon A\to \mathcal N$ and every $\rho>0$ such that $\B(\rho)\setminus \overline{B(\eta\rho)}\subset A$,
\begin{equation}
E_{s,p}(u,A) \leq \left( 1 + \frac{C}{(1-\eta)^{sp+1}} \right)E_{s,p}(u,B(\rho)) +
\left( 1 + \frac{C\eta^{m}}{1-\eta} \right) E_{s,p}(u,A\setminus B(\eta\rho))\text{.}
\end{equation}
\end{lemma}

We use the next lemma for connecting together maps defined on a finite or countably infinite number of small patches:
\begin{lemma}[{\cite[Lemma 2.3]{MonteilVanSchaftigen}}]\label{le:patching}
Let $s\in (0,1)$ and $1\leq p< \infty$. Let $I$ be a finite or countably infinite set, and for each $i\in I$, let $u_i\colon \mathcal M \to \n $ be a measurable map between Riemannian manifolds. 
If there exist $b\in \n $ and a collection $(A_i)_{i\in I}$ of open subsets of $\mathcal M $ such that $A_i\cap A_j = \emptyset$ for $i\neq j$ and $u_i(x)\equiv b$ for $x\in \m \setminus A_i$, then the map \( u \colon \mathcal{M} \to \n \) defined as
\begin{equation}
u(x) =
\begin{cases}
u_i(x) &\text{ if } x\in A_i\text{,} \\
b &\text{otherwise}
\end{cases}
\end{equation}
satisfies
\begin{equation}
E_{s,p}(u, \mathcal M ) \leq 2^p \sum_{i\in I} E_{s,p}(u_i, \mathcal M)\text{.}
\end{equation}
\end{lemma}

The last lemma will allow us to ``open'' maps, that is, to produce a homotopic map which is locally constant, equal to some chosen point. 
Its corollary states that we retain control over the $W^{s,p}$-energy of such an ``opened'' map.
\begin{lemma}[{\cite[Lemma 3.4]{VS20}}]\label{le:holeing}
For every $b \in \n $ and every $\varepsilon>0$, there exists a map $\Theta \in C^1(\n ,\n)$ which is homotopic to the identity and such that $\Theta\equiv b$ in a neighborhood of $b$ and for every $x, y \in \n$, we have
$d_{\n}(\Theta(x), \Theta(y))\leq (1+\varepsilon)d_{\n}(x,y)$.
\end{lemma}

\begin{corollary}\label{cor:energy-holeing}
For any given $u\colon \S^k \to N$, and for every $b\in u(\S^k)\subset \n$ and $\eps>0$, the $\Theta$ from \Cref{le:holeing} is such that
\begin{equation}
E_{s,p}(\Theta\circ u, \S^k ) \leq (1+\varepsilon)^pE_{s,p}(u, \S^k )\text{,}
\end{equation}
$\Theta\circ u \sim u$ and $\Theta\circ u$ is constant on a neighborhood of the fiber $u^{-1}(b)$.
\end{corollary}

\section{{Proof of Theorem \ref{th:main}}}\label{s:main}

This section is exclusively dedicated to the proof of our general existence theorem, namely Theorem~\ref{th:main}.

\begin{proof}[{Proof of \Cref{th:main}}]
    Assume by contradiction that there are finitely many homotopy classes of $\pi_n(\n)$ containing a minimizing $n$-harmonic map, and denote them $\alpha_1,\dotsc, \alpha_\ell$. Additionally, let $d_0$ denote the largest (up to absolute value) non-zero degree associated with those classes:
    \[
    	d_0 \coloneqq \max_{i=1,\dotsc,\ell} |\deg(\alpha_i)|\text{,}
    \]
    and let $E$ be the smallest energy infimum among them:
    \[
    	E \coloneqq \min_{i=1,\dotsc,\ell}\#_{s,n} \alpha_i\text{.}
    \]
    Let now $\beta_k$ be a different homotopy class, satisfying
    \[
    	\deg(\beta_k) = d_k\text{.}
    \]
    By definition of the infimum, we may choose $\beta_k$ such that
    \[
    	\#_{s,n}\beta_k \leq 2\#_{s,n}d_k\text{,}
    \]
    since $\#_{s,n}d_k = \inf_{\deg(\beta) = d_k} \#_{s,n}\beta$. Using the energy identity provided by Theorem~\ref{th:energy-id-general}, we can write
    \begin{equation}
        \#_{s,n} \beta_k  = \sum_{j} \#\alpha_j =
         \sum_{j} \frac{|\deg(\alpha_j)|}{|\deg(\alpha_j)|} \#\alpha_j \geq \frac{E}{d} \sum_j |\deg(\alpha_j)| \geq \frac{E}{d} |d_k|
    \end{equation}
    for some suitable decomposition 
    \[
    	\beta_k = \sum_{j} \alpha_{j}\text{.}
    \]
    In particular, we have used the fact that
    \[
    	d_{k} = \sum_{j} \deg(\alpha_{j})\text{,}
    \]
    w hich is a consequence of the fact that \( \deg \) is a group homomorphism.
    By our choice of $\beta_k$, we obtain
    \begin{equation}
        \#_{s,n}d_k \gtrsim \#_{s,n}\beta_k \gtrsim |d_k|\text{,}
    \end{equation}
    which contradicts the sub-linearity assumption about $\deg$ and finishes the proof. 
\end{proof}

\section{\texorpdfstring{Optimality for $\pi_{4n-1}(\S^{2n})$}{Optimality for spheres}}\label{s:hopfs}
The goal of this section is to prove the following theorem.

\begin{theorem}\label{th:optimalityboth}
Let \(n\in\N\). For every \(s\in(0,1]\), there exists a constant
\(C=C(n,s)>0\) such that
\begin{equation}
    \#_{s,\frac{4n-1}{s}} \alpha
    \leq C|d|^{\frac{4n-1}{4n}}
\end{equation}
for every admissible Hopf degree \(d\) and every class \( \alpha \in \pi_{4n-1}(\S^{2n}) \) with \( \deg_H(\alpha) = d \). Here
\(\#_{s,\frac{4n-1}{s}}\alpha\) denotes the infimum of
\(E_{s,\frac{4n-1}{s}}\) among maps
\(u\in W^{s,\frac{4n-1}{s}}(\S^{4n-1},\S^{2n})\)
with \( [u] = \alpha \), as defined in
\eqref{eq:infenergysn}.
\end{theorem}

By ``admissible'', we mean that, except if \( n = 1,2,4 \), we restrict to even values of \( d \), according to Theorems~\ref{th:hopfdeg1} and~\ref{th:hopfdeg2}.
This theorem in particular provides the asymptotic rate of growth of order \( \frac{4n-1}{4n} \) for the minimal energy for the Hopf invariant.
Our statement is slightly stronger, since it provides an estimate for \emph{every} admissible Hopf degree, not only along a sequence that diverges to infinity.
In particular, an additional difficulty in treating the general case \( n \geq 1 \), compared to \( n = 1 \), is that we have to take into account that \( \pi_{4n-1}(\S^{2n}) \) has torsion elements, which are not seen by the Hopf invariant, but might interfere in our use of the energy identity, changing the homotopy class we are in.
For this reason, an estimate for \emph{every} class \( \alpha \in \pi_{4n-1}(\S^{2n}) \) with \( \deg_H(\alpha) = d \) is stronger than an estimate on \( \#_{s,\frac{4n-1}{s}}d \), as the latter involves an infimum over more maps, since it disregards the torsion class.

The corresponding estimate for \( \#_{s,\frac{4n-1}{s}}d \) was established in
\cite[Proposition 4.1]{Schikorra-VanSchaftingen-Hopf},
where a sketch of the proof is given. Here we provide a detailed
proof, treating the local case \(s=1\) and the nonlocal case
\(0<s<1\) in two separate theorems.

For the nonlocal case, we give an alternative, direct proof that
does not use Gagliardo--Nirenberg interpolation. Although longer
than the interpolation argument in
\cite{Schikorra-VanSchaftingen-Hopf}, this direct approach may be
of independent interest.

\subsection{The local case}

The goal of this section is to prove the case \( s = 1 \) of Theorem~\ref{th:optimalityboth}.
The proof is by direct construction of competitors.
We proceed in two steps.
First we estimate \( \#_{1,4n-1}d \) for \( d = 2k^{2} \), relying on the formula from Proposition~\ref{prop:HopfDegLeftComp}.
Then we explain how to reach any homotopy class, by adding a perturbation with uniformly bounded energy.

\begin{proof}[Proof of Theorem~\ref{th:optimalityboth} when \( s = 1 \)]~

\textsc{Step 1. Proof in the case $d=2k^2$.}

To show the inequality, it suffices to construct a particular map verifying the condition. Let $\Phi\colon \S^{4n-1}\to\S^{2n}$ be a map of Hopf degree \(2\), such as the one constructed in \eqref{eq:hopfdeg2}. For any $v\colon\S^{2n}\to\S^{2n}$ of topological degree $\deg v = k$, by Proposition~\ref{prop:HopfDegLeftComp},
\begin{equation}
\degh(v\circ \Phi) = (\deg v)^2 \degh \Phi = 2k^2\text{.}
\end{equation}
On the other hand, by the chain rule, we can estimate by the chain rule
\begin{equation}\label{eq:local-estimate}
E_{1,4n-1}(v\circ \Phi, \S^{4n-1}) \leq \| \nabla v \|_{L^{\infty}(\S^{2n})}^{4n-1} E_{1,4n-1}(\Phi, \S^{4n-1}) \lesssim \| \nabla v \|_{L^{\infty}(\S^{2n})}^{4n-1}\text{.}
\end{equation}
(Here, the argument is similar to \cite[Proof of Proposition 2.15]{HardtRiviere08}.)

In particular, we may take \( v \) from the construction in \cite[p. 437]{Riviere98}, which we briefly recall here. There exists a constant \( \lambda > 0 \) such that, for any \( k \in \mathbb{N} \), one can find \( k \) disjoint geodesic balls \( \left(B(x_i, \lambda k^{-\frac{1}{2n}})\right)_{i=1}^k \) of radius \( \lambda k^{-\frac{1}{2n}} \) in \( \mathbb{S}^{2n} \). For each \( i \), let \( v_i \colon \mathbb{S}^{2n} \to \mathbb{S}^{2n} \) be a degree-one map that is constant (equal to some fixed \( b \in \mathbb{S}^{2n} \)) outside of \( B(x_i, \lambda k^{-\frac{1}{2n}}) \), and satisfies \( |\nabla v_i| \lesssim k^{\frac{1}{2n}} \). For a detailed construction, we refer to \Cref{a1-3-lmm:top-deg1}. We then define \( v \colon \mathbb{S}^{2n} \to \mathbb{S}^{2n} \) by
\begin{equation}\label{eq:map-construction}
v = \begin{cases}
v_i & \text{on } B(x_i, \lambda k^{-\frac{1}{2n}})\text{,} \\
b & \text{otherwise}\text{.}
\end{cases}
\end{equation}

By construction,
\begin{equation}\label{eq:map-properties1}
\deg v = k \quad \text{and} \quad |\nabla v|\lesssim k^{\frac{1}{2n}}\text{.}
\end{equation}
Thus, coming back to \eqref{eq:local-estimate}, we obtain
\begin{equation}
E_{1,4n-1}(v\circ \Phi, \S^{4n-1}) \lesssim k^{\frac{4n-1}{2n}} = (k^2)^{\frac{4n-1}{4n}}\text{,}
\end{equation}
which shows that
\[
	\#_{1,4n-1}d
	\lesssim
	d^{\frac{4n-1}{4n}}
\] 
in the case of $d=2k^2$.

\textsc{Step 2: Proof for any $d\in\Z$.}

For arbitrary $d\in\Z\setminus\{0\}$, we note that, by \Cref{th:hopfdeg1} and \Cref{th:hopfdeg2}, the class
\begin{equation}
\left\{ u:\S^{4n-1}\to\S^{2n}\ |\ \degh(u) = d \right\}
\end{equation}
is non-empty for all $d$ if and only if $n=1,2,4$ and for even values of $d$ if $n\neq 1,2,4$.
Moreover, by \cite[§ V Proposition 5]{Serre_1951}, \( \pi_{4n-1}(\S^{2n}) = \Z \oplus G \) for some finite (possibly trivial) abelian group \( G \).
For every \( g \in G \), we let \( u_{g} \colon \S^{4n-1} \to \S^{2n} \) be a smooth map with \( [u_{g}] = (0,g) \).
Without loss of generality, we assume that $d>0$, as the case \( d < 0 \) follows by precomposing with an orientation-reversing diffeomorphism of \( \mathbb{S}^{4n-1} \), which changes the sign of the Hopf invariant.  
Choose $k\in \N$ such that
\[
2k^2 < d \leq 2(k+1)^2\text{,}
\]
and thus
\[
0 < d - 2k^2 < 4k < 4\sqrt{d}\text{.}
\]

Given any map $w\colon \S^{4n-1} \to \S^{2n}$ with $\degh(w)=2k^2$, we pick a point $b\in w(\S^{4n-1})\subset \S^{2n}$ and a number $\eps>0$. 
Using Lemma \ref{le:holeing} and Corollary \ref{cor:energy-holeing}, we obtain another map $\tilde{w} \sim w$, which is constant (equal to $b$) in a neighborhood of the fiber $w^{-1}(b)$ and satisfies
\begin{equation}\label{3-3-eq:control0}
E_{1,4n-1}(\tilde{w}, \S^{4n-1} ) \leq (1+\varepsilon)^{4n-1}E_{1,4n-1}(w,\S^{4n-1} )\text{.}
\end{equation}
In particular, $\tilde{w}^{-1}(b)$ contains a non-empty open set and we can find $d-2k^2+1$ disjoint open balls $\left( B(x_i,r)\right)_{i=1}^{d-2k^2+1}$ of positive radius inside it.

Now, either $n=1,2,4$ and we can insert into \( d-2k^{2}\) of those open balls a rescaled copy of the corresponding Hopf map, increasing the energy by at most $(d-2k^2)E$ for some universal constant $E>0$; or the degree is an even number, $d=2d'$, and in this case, we may insert a rescaled mapping of Hopf degree \(2\) into \( d'-k^{2}\) of the balls, leaving the others unchanged. This way, again, we increase the energy by at most $(d'-k^2)E$.
In any case, there is at least one ball left, that we may use to insert a copy of \( u_{g} \) for some \( g \in G \), to suitably adjust the torsion class of our map.
This increases the energy by at most a universal constant, that we can assume to be \( E \), increasing \( E \) if necessary.
The total amount of energy added by this insertion procedure is therefore at most of order
\[
	(d-2k^2+1)E
	\lesssim
	d-2k^{2}+1
	\leq
	4\sqrt{d}\text{.}
\]

We now apply this procedure choosing the map \( w \) such that
\[
	E_{1,4n-1}(w,\S^{4n-1})
	\lesssim
	\left(k^{2}\right)^{\frac{4n-1}{4n}}\text{,}
\]
which is possible by virtue of the first step.
Given \( \alpha \in \pi_{4n-1}(\S^{2n}) \) such that \( \deg_H(\alpha) = d \), choosing suitably the element \( g \) depending on \( w \) and \( \alpha \) to adjust the torsion class, the above insertion procedure yields a map $u$ such that \( [u] = \alpha \) satisfying the energy bound
\begin{equation}
E_{1,4n-1}(u,\S^{4n-1}) \lesssim d^{\frac{4n-1}{4n}} + d^{\frac{1}{2}} \lesssim d^{\frac{4n-1}{4n}}\text{,}
\end{equation}
and thus 
\begin{equation}
\#_{1,4n-1} \alpha \lesssim |d|^{\frac{4n-1}{4n}}\text{.}
\end{equation}
This concludes the proof.
\end{proof}

\subsection{The nonlocal case}\label{s:nonlocalS4n-1}

We now turn to the case \( 0 < s < 1 \).
In this case, Theorem~\ref{th:optimalityboth} can also be deduced from the case \( s = 1 \) by interpolation, relying on the Gagliardo--Nirenberg inequality, as in
\cite{Schikorra-VanSchaftingen-Hopf}. 
Indeed, the
Gagliardo--Nirenberg inequality (see, e.g.,
\cite{BM-Gagliardo}) gives
\[
 [u]_{W^{s,\frac{4n-1}{s}}(\S^{4n-1})}
 \lesssim
 \|u\|_{L^\infty(\S^{4n-1})}^{1-s}
 \|\nabla u\|_{L^{4n-1}(\S^{4n-1})}^{s}
\]
for every
\(u\in W^{1,4n-1}(\S^{4n-1},\S^{2n})\)
and \(0<s<1\). Since \(|u|=1\), raising this inequality
to the power \((4n-1)/s\) yields
\[
 E_{s,\frac{4n-1}{s}}(u, \S^{4n-1})
 \lesssim E_{1,4n-1}(u, \S^{4n-1})\text{.}
\]
Applying this estimate to the competitors \(u_d\) provided by
the local case, we obtain
\[
 E_{s,\frac{4n-1}{s}}(u_d, \S^{4n-1})
 \lesssim E_{1,4n-1}(u_d, \S^{4n-1})
 \lesssim |d|^{\frac{4n-1}{4n}}\text{,}
\]
which proves the desired bound.

We now give a direct proof that avoids interpolation.

\begin{proof}[Proof of Theorem~\ref{th:optimalityboth} when \( 0 < s < 1\)]
We follow the construction used in the case \( s = 1 \), and show that it also works in the setting of fractional Sobolev maps. 
For the sake of conciseness, we introduce $p=\frac {4n-1}{s}$.

\textsc{Step 1. Proof in the case $d=k^2$.}

Again, let $\Phi\colon \S^{4n-1}\to\S^{2n}$ be a map of Hopf degree 2 (see, e.g., \eqref{eq:hopfdeg2}) and $v\colon\S^{2n}\to\S^{2n}$ be a map of Brouwer degree $\deg v = k$. Again, we have
\begin{equation}
\degh(v\circ \Phi) = (\deg v)^2 \degh \Phi = 2k^2\text{.}
\end{equation}
By the co-area formula, see, e.g., \cite[Theorem~1]{Hajlasz00}, we have
\begin{equation}\label{eq:1}
\begin{split}
&E_{s,p}(v\circ \Phi, \S^{4n-1} ) \\
&\quad\quad\quad= \int_{\S^{4n-1} } \int_{\S^{4n-1}} \frac{|v\circ \Phi(x) - v\circ \Phi(y)|^p}{|x-y|^{4n-1+sp}} \dx\dy \\
&\quad\quad\quad= \int_{\S^{4n-1} } \int_{\S^{4n-1} } \frac{|v\circ \Phi(x) - v\circ \Phi(y)|^p}{|x-y|^{4n-1+sp}} \frac{|\Phi^*\omega_{\S^{2n}}|}{|\Phi^*\omega_{\S^{2n}}|}\frac{|\Phi^*\omega_{\S^{2n}}|}{|\Phi^*\omega_{\S^{2n}}|}\dx\dy\\
&\quad\quad\quad= C\int_{\S^{2n}} \int_{\S^{2n}} \left(\int_{\Phi^{-1}(z_2)}\int_{\Phi^{-1}(z_1)}\frac{|v\circ \Phi(x) - v\circ \Phi(y)|^p}{|x-y|^{4n-1+sp}}\dx\dy \right) \dz_1\dz_2\\
&\quad\quad\quad=C\int_{\S^{2n}} \int_{\S^{2n}} |v(z_1) - v(z_2)|^p \left(\int_{\Phi^{-1}(z_2)}\int_{\Phi^{-1}(z_1)}\frac{1}{|x-y|^{4n-1+sp}}\dx\dy \right) \dz_1\dz_2\text{.}
\end{split}
\end{equation}
The second to last equality follows since $|\Phi^*\omega_{\S^{2n}}|$ is bounded from below, away from zero. 

To integrate the kernel $|x-y|^{-4n+1-sp}$ along the $(2n-1)$-dimensional fibres, we rewrite the integral using the Cavalieri principle as
\begin{equation}\label{eq:2}
\begin{split}
\int_{\Phi^{-1}(z_2)}\int_{\Phi^{-1}(z_1)}\frac{1}{|x-y|^{4n-1+sp}}\dx\dy =
\int_{\Phi^{-1}(z_2)}\left(\int_{S}\frac{1}{|t|^{4n-1+sp}}\dif\mathcal{H}^{2n-1}(t) \right) \dy  \\
= \int_{\Phi^{-1}(z_2)}\left(\int_{0}^{\infty}(4n-1+sp) \tau^{4n-2+sp} \mathcal{H}^{2n-1}\left(S \cap \left\{t:\ \frac{1}{|t|}> \tau\right\}\right) \dif\tau \right) \dy\text{,}
\end{split}
\end{equation}
where $S = -y + \Phi^{-1}(z_1)$ are two linked $2n-1$ dimensional spheres of radii $r\leq 1$. 

Now, notice that
\begin{equation}\label{3-1-eq:2pitau}
\mathcal{H}^{2n-1}\left(S \cap \left\{t:\ \frac{1}{|t|}> \tau\right\}\right) =
\mathcal{H}^{2n-1}\left(S \cap \left\{t:\ |t|< \frac{1}{\tau}\right\}\right) \lesssim 
\tau^{-(2n-1)}\text{.}
\end{equation}
Additionally, observe that, since $|z_1 - z_2|\leq \text{Lip}(\Phi) |x - y| = \text{Lip}(\Phi)|t|$, then
\begin{equation}\label{3-1-eq:bound}
\mathcal{H}^{2n-1}\left(S \cap \left\{t:\ |t|< \frac{1}{\tau}\right\}\right) = 0 \text{ for } \tau \geq \frac{\text{Lip}(\Phi)}{|z_1 - z_2|}\text{.}
\end{equation}
Putting \eqref{3-1-eq:2pitau} and \eqref{3-1-eq:bound} together, we obtain
\begin{equation}\label{eq:3}
\begin{split}
&\int_{0}^{\infty}(4n-1+sp) \tau^{4n-2+sp} \mathcal{H}^{2n-1}\left(S \cap \left\{t:\ \frac{1}{|t|}> \tau\right\}\right) \dif\tau \\
&\lesssim \int_0^{\frac{\text{Lip}(\Phi)}{|z_1-z_2|}}(4n-1+sp)\tau^{4n-2+sp}\tau^{-2n+1} \dif\tau
\lesssim \frac{1}{|z_1 - z_2|^{2n+sp}}\text{.}
\end{split}
\end{equation}
Consequently, combining \eqref{eq:1} with \eqref{eq:2} and \eqref{eq:3}, we obtain
\begin{equation}
E_{s,p}(v\circ \Phi,\S^{4n-1})\lesssim \int_{\S^{2n}}\int_{\S^{2n}} \frac{|v(z_1)-v(z_2)|^p}{|z_1 - z_2|^{2n+sp}}\dif z_1 \dif z_2 = E_{s,p}(v,\S^{2n})\text{.}
\end{equation}
Thus, we conclude that
\begin{equation}\label{3-1-eq:poczatek}
\#_{s,p}2k^2 \leq E_{s,p}(v\circ \Phi, \S^{4n-1} ) \lesssim E_{s,p}(v,\S^{2n})
\end{equation}
for any $v\colon\S^{2n}\to\S^{2n}$ of degree $k$.

In particular, we may take \( v \) as in Step~1 of the proof of the case \( s = 1 \). 
Recall that, by construction,
\begin{equation}\label{eq:map-properties2}
\deg v = k \quad \text{and} \quad |\nabla v|\lesssim {k}^{\frac{1}{2n}}\text{.}
\end{equation}

Employing Lemma \ref{le:patching}, we write
\begin{equation}\label{eq:firstvestimate}
E_{s,p}(v, \S^{2n}) \leq 2^p\sum_{i=1}^kE_{s,p}(v_i, \S^{2n}).
\end{equation}
Additionally, by using Lemma \ref{le:gluing} for $\eta = 1/2$ and $\rho = 2\lambda k^{-\frac{1}{2n}}$, we obtain
\begin{equation}
\begin{split}
E_{s,p}\left(v_i, \S^{2n}\right) &\lesssim
E_{s,p}\left(v_i, \bB(x_i, 2\lambda k^{-\frac{1}{2n}})\right) +
E_{s,p}\left(v_i,\S^{2n}\setminus \bB(x_i, \lambda k^{-\frac{1}{2n}})\right) \\
&\lesssim E_{s,p}\left(v_i,  \bB(x_i, 2\lambda k^{-\frac{1}{2n}})\right)\text{,}
\end{split}
\end{equation}
where the last inequality is a consequence of $v_i$ being constant on the set $\S^{2n}\setminus \bB(x_i, \lambda k^{-\frac{1}{2n}})$. It remains to estimate the localized energy:
\begin{equation}\label{3-1-eq:koniec}
\begin{split}
E_{s,p}\left (v_i,  \bB(x_i, 2\lambda k^{-\frac{1}{2n}})\right) &=
\int_{ \bB(x_i, 2\lambda k^{-\frac{1}{2n}})}
\int_{ \bB(x_i, 2\lambda k^{-\frac{1}{2n}})}
\frac{|v_i(x) - v_i(y)|^p}{|x-y|^{2n+sp}}\dx\dy\\
&=
\int_{ \bB(x_i, 2\lambda k^{-\frac{1}{2n}})}
\int_{ \bB(x_i, 2\lambda k^{-\frac{1}{2n}})}
\frac{|v_i(x) - v_i(y)|^p}{|x-y|^{p}} \frac{1}{|x-y|^{2n+(s-1)p}} \dx\dy \\
&\lesssim k^{\frac{p}{2n}}
\int_{ \bB(x_i, 2\lambda k^{-\frac{1}{2n}})}
\int_{ \bB(x_i, 2\lambda k^{-\frac{1}{2n}})}
\frac{1}{|x-y|^{2n+(s-1)p}}\dx\dy \\
&\leq k^{\frac{p}{2n}}
\int_{ \bB(x_i, 2\lambda k^{-\frac{1}{2n}})}
\int_{\bB(0, 4\lambda k^{-\frac{1}{2n}})}
\frac{1}{|z|^{2n+(s-1)p}} \dz \dy\\
&\lesssim k^{\frac p {2n}} \left( k^{-\frac{1}{2n}} \right)^{2n} \left(k^{-\frac{1}{2n}}\right) ^{-(s-1)p} = 
k^{\frac{sp}{2n}-1} = k^{\frac{4n-1}{4n} -1}\text{.}
\end{split}
\end{equation}
Gathering \eqref{3-1-eq:poczatek} and \eqref{eq:firstvestimate}-\eqref{3-1-eq:koniec} along with the fact that $sp=3$, we obtain
\begin{equation}
\#_{s, p}2k^2\lesssim E_{s,p}(v, \S^{2n}) \lesssim \sum_{i=1}^k E_{s,p}(v_i, \S^{2n}) \lesssim k^{\frac{4n-1}{2n}} = \left(k^2\right)^{{\frac{4n-1}{4n}}}\text{.}
\end{equation}
This concludes the first part of the proof.

\textsc{Step 2: Proof for any $d\in\Z$.}
To prove the claim for arbitrary Hopf degree, we perform the same gluing argument as for the case \( s = 1 \). 
Fix any \( d \in \mathbb{Z} \setminus \{0\} \) and \( \alpha \in \pi_{4n-1}(\S^{2n}) \) such that \( \deg_H(\alpha) = d \). 
Without loss of generality, assume \( d > 0 \). 
As before, choose \( k \in \mathbb{N} \) such that
\[
2k^2 < d \leq 2(k+1)^2\text{.}
\]
Then,
\[
0 < d - 2k^2 < 4k < 4\sqrt{d}\text{.}
\]

Given any map $w\colon \S^{4n-1} \to \S^{2n}$ with $\degh(w)=2k^2$, pick a point $b\in w(\S^{4n-1})\subset \S^{2n}$ and a number $\eps>0$.
Using Lemma \ref{le:holeing} and Corollary \ref{cor:energy-holeing}, we obtain another map $\tilde{w} \sim w$, which is constant (equal to $b$) in a neighbourhood of the fibre $w^{-1}(b)$ and satisfies
\begin{equation}\label{3-3-eq:control}
E_{s,p}(\tilde{w}, \S^{4n-1} ) \leq (1+\varepsilon)^pE_{s,p}(w,\S^{4n-1} )\text{.}
\end{equation}
As such, $\tilde{w}^{-1}(b)$ contains a non-empty open set and we can find $d-2k^2+1$ disjoint open balls $\left( B(x_i,r)\right)_{i=1}^{d-2k^2+1}$ of positive radius inside it.

As before, we glue inside these balls rescaled copies of the Hopf map (or, if such a map is not available, of mappings of Hopf degree \( 2 \) such as the one from \eqref{eq:hopfdeg2}, inside \( \frac{d-2k^{2}}{2} \) of these balls), denoted $h_{x_i,r}$, with respective supports in $B(x_i,r)$ (i.e., constantly equal to $b$ outside these balls), and such that $E_{s,p}(h_{x_i,r},\S^{4n-1} ) \leq E$.
For \( i = d-2k^2+1 \), we let instead \( h_{x_{i},r} \) be a map such that \( [h_{x_{i},r}] = (0,g) \) with \( g \in G \) to be suitably chosen, satisfying the same support condition and \( E_{s,p}(h_{x_i,r},\S^{4n-1} ) \leq E \).
We define $u\colon \S^{4n-1}\to\S^{2n}$  by
\begin{equation}
u(x) = \begin{cases}
h_{x_i,r}(x) &\text{if } x\in  B(x_i,r)\text{,}\\
\tilde{w}(x) &\text{otherwise.}
\end{cases}
\end{equation}
If \( g \) is suitably chosen, also depending on \( w \), we have \( [u] = \alpha \). 
Additionally, we observe that the sets
\begin{equation}
A_0 = \S^{4n-1} \setminus \tilde{w}^{-1}(b)\text{,} \quad A_i =  B(x_i,r)
\end{equation}
and the maps $\tilde{w}$ for \( i = 0 \) and $h_{x_i,r}$ for each of \( i = 1, \dotsc, d-2k^{2}+1 \) satisfy the assumptions of Lemma~\ref{le:patching}.
Therefore,
\begin{equation}
\begin{split}
\#_{s,p}\alpha\leq E_{s,p}(u,\S^{4n-1} )&\lesssim
E_{s,p}(\tilde{w}, \S^{4n-1} ) + (d-2k^2+1)E \\
&\lesssim (1+\eps)^p E_{s,p}(w, \S^{4n-1} ) + (d-k^2+1)\text{,}
\end{split}
\end{equation}
where the last inequality relies on \eqref{3-3-eq:control}. 
Passing to the infimum on the right-hand side, we obtain
\begin{equation}
\#_{s,p}\alpha \lesssim \#_{s,p}2k^2 + (d-k^2+1) \lesssim d^{\frac{4n-1}{4n}} + d^{1/2} \lesssim d^{\frac{4n-1}{4n}}\text{.}
\end{equation}
This concludes the proof.
\end{proof}

Let us note that the full optimality of the exponent $\frac{4n-1}{4n}$ in the case of $E_{s,\frac{4n-1}{s}}$ is still open. 
In the fractional setting a counterpart to \eqref{eq:hopf-estimate} is possible and was established in \cite[Theorem 1.1]{Schikorra-VanSchaftingen-Hopf} for maps in $\pi_{4n-1}(\S^{2n})$ for the range $s\in[1-\frac{1}{4n}, 1)$.
That is, if \( s\in[1-\frac{1}{4n}, 1) \), then
\begin{equation}\label{eq:fractioanlhopfestimate}
\abs{\deg_H(f)} \lesssim [f]^{\frac{4n}{s}}_{W^{s,\frac{4n-1}{s}}(\S^{4n-1})}\text{.}
\end{equation}
See also~\cite{ParkSchikorra} for related results for more general integral invariants, based on the notions introduced by Hardt and Rivière \cite{HardtRiviere08}. 
It remains an open problem to establish whether or not this estimate holds in the full range $s\in(0,1)$.

\subsection{Existence of infinitely many minimizers}

With Theorem~\ref{th:optimalityboth} at hand, we are now in position to deduce Corollary~\ref{th:higher-hopf}.

\begin{proof}[Proof of Corollary~\ref{th:higher-hopf}]
	Since \( \#_{s,\frac{4n-1}{s}}d \leq \#_{s,\frac{4n-1}{s}}\alpha \) whenever \( \deg_{H}(\alpha) = d \), Theorem~\ref{th:optimalityboth} implies that
	\[
		\#_{s,\frac{4n-1}{s}}d
		\lesssim
		\abs{d}^{\frac{4n-1}{4n}}\text{,}
	\]
	and in particular, since \( \frac{4n-1}{4n} < 1 \),
	\[
		\lim_{d \to \infty} \frac{1}{d}\inf_{\deg_{H}(u)=d}E_{s,\frac{4n-1}{s}}(u) = 0\text{.}
	\]
	The conclusion follows directly from Theorem~\ref{th:main}.
\end{proof}

\begin{proof}[Sketch of a proof of \Cref{rem:RPn}]
The construction in \Cref{th:optimalityboth} also yields
the required energy bounds for maps into $\mathbb{RP}^{2n}$.
Endow $\mathbb{RP}^{2n}$ with its standard quotient metric and let
\[
\pi\colon \S^{2n}\to\mathbb{RP}^{2n}
\]
be the canonical covering projection; see
\cite[Example~1.43]{Hatcher}.
For each admissible Hopf degree $d$, let $u_d$ be the map
constructed in the proof of \Cref{th:optimalityboth}, and set
$v_d=\pi\circ u_d$.
Since $\pi$ is a local isometry,
\[
|\nabla v_d|=|\nabla u_d|
\quad\text{a.e. on }\S^{4n-1}\text{,}
\]
and hence
\[
E_{1,4n-1}(v_d)=E_{1,4n-1}(u_d)\text{.}
\]
For $s\in(0,1)$, the projection $\pi$ is globally Lipschitz with some
constant $L$, so
\[
E_{s,\frac{4n-1}{s}}(v_d)
\le
L^{\frac{4n-1}{s}}
E_{s,\frac{4n-1}{s}}(u_d)\text{.}
\]
Moreover, the covering projection induces an isomorphism
\[
\pi_*\colon
\pi_{4n-1}(\S^{2n})
\to
\pi_{4n-1}(\mathbb{RP}^{2n})\text{;}
\]
see \cite[Proposition~4.1]{Hatcher}.
The two lifts of a continuous map into $\mathbb{RP}^{2n}$
differ by postcomposition with the antipodal map, which
preserves the Hopf invariant. Consequently, the Hopf invariant
of a lift is well defined on free homotopy classes of maps
into $\mathbb{RP}^{2n}$ and is additive under free homotopy
decompositions.

To conclude, we use the energy identity in
\Cref{th:energy-id-general} (as in proof of \Cref{th:main}), formulated in terms of free
homotopy decompositions to account for the nontrivial
fundamental group of $\mathbb{RP}^{2n}$; see
\cite[Proposition~2.2]{VS20}.
For $s=1$, this formulation follows from
\cite{Duzaar-Kuwert}.
For $s\in(0,1)$, the proof of
\cite[Theorem~15]{MS-stability} carries over using free
homotopy decompositions. Indeed,
$\pi_1(\mathbb{RP}^{2n})\simeq\mathbb Z/2\mathbb Z$,
so every orbit of its action on
$\pi_{4n-1}(\mathbb{RP}^{2n})$ is finite.
Thus \cite[Theorem~1.4]{VS20} provides the finiteness of
free homotopy classes with bounded energy required in that
proof. Combining this energy identity with the preceding
sublinear energy bound yields the assertion by the same
argument as for sphere targets.
\end{proof}

\section{\texorpdfstring{Optimality for $\pi_{2n+1}(\CP^n)$}{Optimality for projective spaces}}\label{s:cpn}
In this section, we give another family of examples, involving maps into complex projective spaces, where the exponent associated to a real homotopy invariant is optimal. 
This results in the existence of infinitely many minimizers of the corresponding energy functional, i.e., \Cref{cor:cpn}. The homotopy invariant in question is given, for a smooth map $u\colon \S^{2n+1}\to\CP^n$, by the integral
\[
	 \deg([u]) = \int_{\S^{2n+1}} \dext^{-1}(u^*\omega)\wedge u^*\omega^n\text{,}
\]
where $\omega\in\Omega^2{(\CP^n)}$ is the K\"ahler form and \( \omega^{n} = \omega \wedge \dotsb \wedge \omega \). 
Since $\pi_{2n+1}(\CP^n)\simeq \Z$, this invariant computes the homotopy class of $u$. 
The associated quantitative estimate reads
\[
|\deg(u)| \lesssim \left( \int_{\S^{2n+1}} |\nabla u|^{2n+1} \right)^{\frac{2n+2}{2n+1}}\text{.}
\]

The main result of this section is the following theorem, which provides the reverse inequality, also in the fractional setting:
\begin{theorem}\label{th:cpn}
	Let \( n \in \N \).
	For every \( s \in (0,1] \), there exists a constant $C=C(n,s) > 0$ such that
	\begin{equation}\label{eq:infboundforCPn}
		\#_{s,\frac{2n+1}{s}}d \leq C|d|^{\frac{2n+1}{2n+2}}
	\end{equation}
	for every \( d \in \Z \).
\end{theorem}

Once again, we divide the proof between the case \( s = 1 \) and the case \( 0 < s < 1 \).

\subsection{The local case}

In this section, we prove the case \( s = 1 \) of Theorem~\ref{th:cpn}.
Let us mention that, for \( \CP^{2} \), the result was already stated in \cite[Proposition~2.15]{HardtRiviere08}. 
Here we give a proof valid in arbitrary dimension $n$.

Our construction uses the following consequence of the quantitative
nullhomotopy theorem of Berdnikov and Manin.

\begin{lemma}[Quantitative nullhomotopy]\label{lem:filling}
For fixed integers $n\geq1$ and $m\geq1$, there is a constant
$A_{n,m}$ such that every nullhomotopic $L$-Lipschitz map
\[
 g\colon \S^m\longrightarrow\CP^n
\]
extends to a map $G\colon \mathbb B^{m+1}\to\CP^n$ with
\[
 \Lip(G)\leq A_{n,m}(L+1)\text{,}
\]
with the sphere and disk carrying their standard metrics.
\end{lemma}

The above lemma can be deduced directly from \cite[Theorem A, (iii)]{Manin-Berdnikov-2022}. 
The fact that this theorem applies to \( \CP^{n} \) --- in the language of \cite{Manin-Berdnikov-2022}, this means that \( \CP^{n} \) is \emph{scalable} --- follows from the fact that \( \CP^{n} \) is a symmetric space; see \cite[Table~1]{Manin-Berdnikov-2022}.

\begin{proof}[Proof of \Cref{th:cpn} when \( s = 1 \)]
As before, it suffices to construct any map $f\colon \S^{2n+1}\to\CP^n$ whose energy is bounded by the right-hand side of \eqref{eq:infboundforCPn}. We shall first prove the proposition for $d=k^{n+1}$. 

\textsc{{Step 1. Maps of degree $k^{n+1}$.}}

Let $k\in\Z$, $n\in\N$, and let $\omega\in\Omega^2(\CP^n)$ be the standard K\"{a}hler form. Let $\Pi\colon \S^{2n+1}\to\CP^n$ be the principal bundle projection, given in complex homogeneous coordinates by the formula
\begin{equation}\label{eq:submersion}
    \Pi(z_0,z_1,\dotsc, z_n) = [z_0:z_1:\dotsb:z_n]\text{.}
\end{equation}
By the homotopy long exact sequence of the fibration $\S^1\hookrightarrow \S^{2n+1}\to \CP^n$, we know that the homotopy class of $\Pi$ corresponds to $1$ in $\pi_{2n+1}(\CP^n)\simeq \Z$.
(Actually, the map \( \Pi \) induces an isomorphism \( \pi_{2n+1}(\CP^n)\simeq\pi_{2n+1}(\S^{2n+1}) \simeq \Z \); see \cite[Theorem~4.41 and Example~4.44]{Hatcher}.) 
Therefore, we have (up to normalization)
\[
\deg(\Pi) = \int_{\S^{2n+1}} \eta\ \wedge \Pi^*\omega^n = 1\text{,}
\]
where $\eta$ is a 1-form on $\S^{2n+1}$ satisfying $d\eta = \Pi^*\omega$.
The integral formula for the degree on \( \pi_{2n+1}(\CP^{n}) = \Z \) is established by rational homotopy theory, see Example~\ref{ex:cpn_degre_integral}.

We will construct a family of smooth maps $f^k_n\colon \CP^n\to\CP^n$ satisfying
\begin{enumerate}[label=(\arabic*), ref=(\arabic*)]
\item \label{it:fhashomega} $[(f^k_n)^\#\omega] = k[\omega]$,
\item \label{it:nablaf} $\|\nabla f^k_n\|_{\infty} \lesssim \sqrt{k}$.
\end{enumerate}
Before proceeding with the construction, let us see why this allows us to conclude this step. 
On the one hand, we observe that
\[
\deg(f^k_n\circ \Pi) = \int_{\S^{2n+1}} k\eta\ \wedge \Pi^*(k\omega) = k^{n+1}\text{.}
\]
On the other hand, we estimate the energy of \( f^{k}_{n} \) as
\begin{equation}
\begin{split}
	\int_{\S^{2n+1}} |\nabla (f^k_n\circ \Pi)|^{2n+1}
    &\leq \|\nabla f^k_n\|_{L^\infty(\CP^n)}^{2n+1}\int_{\S^{2n+1}}|\nabla \Pi|^{2n+1}\\
    &\lesssim k^{\frac{2n+1}{2}} = (k^{n+1})^{\frac{2n+1}{2n+2}}\text{.}
\end{split}
\end{equation}
This shows that \( f_{n}^{k} \circ \Pi \) is an admissible competitor satisfying the required energy bound.
 
We turn now to the construction of the family of maps $f_n^k\colon \CP^n\to\CP^n$. 
We shall proceed by induction on $n$. 

\paragraph{\textit{Inductive construction of $f_n^k$}}
The base case of $n=1$ is essentially covered by Lemma \ref{a1-3-lmm:top-deg1}, since $\CP^1 \simeq \S^2$. It suffices to glue $k$ copies of the map from Lemma \ref{a1-3-lmm:top-deg1} on balls of radius proportional to $\frac{1}{\sqrt{k}}$. By construction, this gives a mapping of degree $k$, which consequently gives $[(f^k_1)^\#\omega] = k[\omega]$, which satisfies the gradient bound
\[
	\|\nabla f^k_1\|_{L^\infty(\CP^1)} \lesssim \sqrt{k}\text{,}
\]
again by Lemma \ref{a1-3-lmm:top-deg1}.

Assume now that $f_{n-1}^k$ has been constructed. Take the hyperplane $P=\{ z_0=0 \}$. Observe that
\[
\CP^{n-1} \simeq P \subset \CP^n\text{,} 
\]
and denote by $i\colon P\hookrightarrow\CP^n$ the inclusion. Decompose $\CP^n=N\cup B$, where
\begin{align*}
 N&=\left\{[z_0:z_1:\dotsb:z_n]\colon
       |z_0|^2\leq\sum_{j=1}^n|z_j|^2\right\},\\
 B&=\left\{[1:w_1:\dotsb:w_n]\colon
       \sum_{j=1}^n|w_j|^2\leq1\right\}\text{.}
\end{align*}
By the tubular neighbourhood theorem, the set $N$ is a fixed disk-bundle neighbourhood of $P$ and admits the
smooth projection
\[
 \pi\colon N\longrightarrow P\text{,}
 \qquad
 \pi([z_0:z_1:\dotsb:z_n])=[0:z_1:\dotsb:z_n]\text{.}
\]
Because $N$ is compact, this projection has a bounded differential.
Define
\[
 F_{n,k}=i\circ f_{n-1,k}\circ\pi\quad\text{on }N\text{.}
\]
All metric comparison constants here depend only on $n$, so
\[
 \Lip(F_{n,k})\lesssim \sqrt{k}\text{.}
\]

The affine parametrization identifies $B$ smoothly with the closed
unit ball $\overline{\mathbb B}^{2n}$, with a fixed bi-Lipschitz comparison of
metrics. The boundary map of \( F_{n,k} \) is
\[
 g_{n,k}\colon \S^{2n-1}\longrightarrow\CP^n\text{,}
 \qquad
 g_{n,k}(w)=i\bigl(f_{n-1,k}([w_1:\dotsb:w_n])\bigr)\text{.}
\]
It satisfies $\Lip(g_{n,k})\lesssim \sqrt{k}$.
The long exact homotopy sequence of the Hopf fibration
\[
 \S^1\hookrightarrow \S^{2n+1}\rightarrow\CP^n
\]
gives
\[
 \pi_{2n-1}(\CP^n)=0\qquad(n\geq2)\text{,}
\]
since $\pi_{2n-1}(\S^{2n+1})=0$ and $\pi_{2n-2}(\S^1)=0$.
Thus, $g_{n,k}$ is nullhomotopic. Lemma~\ref{lem:filling} gives an
extension $G_{n,k}\colon \mathbb B^{2n}\to\CP^n$ with
\[
 \Lip(G_{n,k})\lesssim \sqrt{k}+1
 \lesssim \sqrt{k}\text{.}
\]

Glue $F_{n,k}$ and $G_{n,k}$ along $\partial N=\partial B$.
The maps agree there, so the resulting glued map $f_{n,k}\colon \CP^n\to\CP^n$
is Lipschitz with
\[
 \Lip(f_{n,k})\lesssim \sqrt{k}\text{.}
\]
Its restriction to $P$ is $i\circ f_{n-1,k}$.
Since $i^*\omega_n=\omega_{n-1}$ --- here we add an index to distinguish the Kähler forms of \( \CP^{n-1} \) and \( \CP^{n} \) --- and
$i^*\colon H_{\textnormal{dR}}^2(\CP^n)\to H_{\textnormal{dR}}^2(\CP^{n-1})$ is an isomorphism, we obtain
\[
 i^*(f_{n,k}^*\omega_n)
 =f_{n-1,k}^*(i^*\omega_n)
 =k \omega_{n-1},
 \qquad
 \text{and hence}
 \qquad f_{n,k}^*\omega_n=k \omega_n\text{.}
\]
This completes the Lipschitz induction.

Finally, a Lipschitz map between fixed compact smooth manifolds can
be approximated uniformly by a homotopic smooth map whose derivative
is bounded by a fixed constant times its Lipschitz constant plus \( 1 \).
To see this, embed the target into a Euclidean space, smooth in fixed
source charts using a partition of unity, and project back from a
fixed tubular neighborhood. Sufficient uniform closeness gives a
homotopy through the same tubular projection. Thus, we obtain a smooth
map $f_{n,k}$ with unchanged cohomology action and
\[
 \|\nabla f_{n,k}\|_{\infty}\lesssim \sqrt{k}\text{,}
\]
which proves the asserted upper bound.

\textsc{{Step 2. Maps of any degree.}} Without loss of generality, we can assume $d>0$ and choose $k$ such that
\[
k^{n+1} \le d < (k+1)^{n+1}\text{.}
\]
A rough estimate yields
\[
d - k^{n+1} \lesssim k^n = (k^{n+1})^{\frac{n}{n+1}} \leq d^{\frac{2n}{2n+2}}\text{.}
\]
Now, similarly to the proof of \Cref{th:optimalityboth}, use Lemma \ref{le:holeing} and Corollary \ref{cor:energy-holeing} to ``open'' a map constructed in the previous step, increasing the energy by at most a multiplicative constant. In this way, we produce a map $\tilde{f}^k_n$ homotopic to $f^k_n$ such that, for some $b\in \CP^n$, the preimage $(\tilde{f}^k_n)^{-1}(b)$ contains a non-empty open set and we can find $d-k^{n+1}$ disjoint open balls $\left( B(x_\ell,r)\right)_{\ell=1}^{d-k^{n+1}}$ of positive radius inside it. 

We can now perform a gluing procedure, inserting into each of those open balls a rescaled copy of $\Pi$, increasing the energy by at most $(d-k^{n+1})E$ for some universal constant $E>0$. We obtain a map $f$ with $\deg(f) = d$, satisfying the energy bound
\begin{equation}
E_{1,2n+1}(f,\S^{2n+1}) \lesssim d^{\frac{2n+1}{2n+2}} + d^{\frac{2n}{2n+2}} \lesssim d^{\frac{2n+1}{2n+2}}\text{.}
\end{equation}
Thus,
\begin{equation}
\#_{1,2n+1} d \leq C|d|^{\frac{2n+1}{2n+2}}\text{,}
\end{equation}
which concludes the proof.
\end{proof}

\subsection{The nonlocal case} 

Similarly as in \Cref{s:nonlocalS4n-1}, we may deduce the nonlocal case from the local one by interpolation. 

\begin{proof}[Proof of Theorem~\ref{th:cpn} when \( 0 < s < 1\)]
By the Gagliardo--Nirenberg inequality (see, e.g.,
\cite{BM-Gagliardo}), we have
\[
 [u]_{W^{s,\frac{2n+1}{s}}(\S^{2n+1})}
 \lesssim
 \|u\|_{L^\infty(\S^{2n+1})}^{1-s}
 \|\nabla u\|_{L^{2n+1}(\S^{2n+1})}^{s}
\]
for every
\(u\in W^{1,2n+1}(\S^{2n+1},\CP^n)\)
and \(0<s<1\). Since \(\CP^n\) is compact, raising this inequality
to the power \((2n+1)/s\) yields
\[
 E_{s,\frac{2n+1}{s}}(u, \S^{2n+1})
 \lesssim E_{1,2n+1}(u, \S^{2n+1})\text{.}
\]
Applying this estimate to the competitors \(u_d\) provided by
the case \( s = 1 \), we obtain
\[
 E_{s,\frac{2n+1}{s}}(u_d, \S^{2n+1})
 \lesssim E_{1,2n+1}(u_d, \S^{2n+1})
 \lesssim |d|^{\frac{2n+1}{2n+2}}\text{,}
\]
which proves the desired bound and concludes the proof.
\end{proof}

\subsection{Existence of infinitely many minimizers}

We now exploit Theorem~\ref{th:cpn} to deduce Corollary~\ref{cor:cpn}.

\begin{proof}[Proof of Corollary~\ref{cor:cpn}]
	As a direct consequence of Theorem~\ref{th:cpn}, still denoting by \( \deg \) the integral invariant on \( \pi_{2n+1}(\CP^{n}) \), we have
	\[
		\lim_{d \to \infty}\frac{1}{d}\inf_{\deg(u)=d} E_{s,\frac{2n+1}{s}}(u)
		=
		0\text{.}
	\]
	The conclusion is then a direct consequence of Theorem~\ref{th:main}.
\end{proof}

\section{Existence of minimizers in specific homotopy classes}\label{s:higher-degree}

In this section, we move toward the proof of Theorem~\ref{pr:hopf234}, which is about identifying explicit homotopy classes containing an energy minimizer.

\subsection{Lower energy bounds on homotopy classes}

We first present a straightforward generalization of \cite[Section 6]{Wang-Zhang}, which yields a lower bound for the infimum of the energy $E_{1,4n-1}$ among maps of fixed degree $d\in \Z$. 
\begin{proposition}
\label{pr:reverseenergyestimaten}
	For every $n\in\N$ and every $d\in\Z$ realized as the Hopf degree of a smooth map from $\S^{4n-1}$ to $\S^{2n}$, we have
	\[
		\#_{1,4n-1} d
		\geq
		C_n\abs{d}^{\frac{4n-1}{4n}}\text{,}
	\]
	where $C_n = (2n)^{\frac{8n^{2}+2n-1}{4n}}\abs{\S^{4n-1}}^{\frac{1}{4n}}\abs{\S^{2n}}^{\frac{4n-1}{2n}}$. We note that for $n=1$ we have $C_1 = \#_{1,3}1 = 32\sqrt{2}\pi^2$.
\end{proposition}
\begin{proof}
We follow \cite[Proof of Theorem 1.4]{Wang-Zhang}.
	Let \( u \in W^{1,4n-1}(\S^{4n-1}, \S^{2n}) \) satisfy \( \deg_H(u) = d \) for some \( d > 0 \) (without loss of generality).
	Let \( \eta \in \Omega^{2n-1}(\S^{4n-1}) \) be such that \( \dext\eta = u^{\ast}\omega_{\S^{2n}} \).
	As the (Frobenius) norm \( \abs{\nabla u} \) is the square root of the sum of the squared singular values of \( \nabla u \) and as \( \abs{u^{\ast}\omega_{\S^{2n}}} \) is the product of those singular values, the Arithmetic Mean–Geometric Mean inequality implies that
	\[
		\frac{\abs{\nabla u}^{2}}{2n}
		\geq
		\abs{\dext\eta}^{\frac{1}{n}}\text{.}
	\]
	Raising to the power \( \frac{4n-1}{2} \) and integrating over \( \S^{4n-1} \), we find
	\[
		\int_{\S^{4n-1}} \abs{\nabla u}^{4n-1}
		\geq
		(2n)^{\frac{4n-1}{2}}\int_{\S^{4n-1}} \abs{\dext\eta}^{\frac{4n-1}{2n}}\text{.}
	\]
	We now apply \cite[Theorem~1.1]{Wang-Zhang}, and find
	\[
		\int_{\S^{4n-1}} \abs{\nabla u}^{4n-1}
		\geq
		(2n)^{\frac{4n-1}{2}}(2n)^{\frac{4n-1}{4n}}\abs{\S^{4n-1}}^{\frac{1}{4n}}\brac{\int_{\S^{4n-1}}\langle \curl\eta,\eta \rangle}^{\frac{4n-1}{4n}}\text{.}
	\]
	Since \( \deg_H(u) = d \), we obtain
	\[
		\int_{\S^{4n-1}} \abs{\nabla u}^{4n-1}
		\geq
		(2n)^{\frac{8n^{2}+2n-1}{4n}}\abs{\S^{4n-1}}^{\frac{1}{4n}}\abs{\S^{2n}}^{\frac{4n-1}{2n}}d^{\frac{4n-1}{4n}}\text{,}
	\]
	which concludes the proof.
\end{proof}

We now exploit Proposition~\ref{pr:reverseenergyestimaten}, combined with the energy identity stated as \Cref{th:energy-id-general}, to deduce the existence of higher degree minimizers provided that we find a map of sufficiently small energy (but not neccessarily minimal).

\begin{proposition}\label{pr:energyboundforexistencen}
	Let \( n \in \N \) and let $d\in\Z$, $d\ge 2$ be an integer realized as the Hopf degree of a smooth map from $\S^{4n-1}$ to $\S^{2n}$.
	If
	\begin{equation}\label{eq:energyboundforexistencen}
		\#_{1,4n-1} d < (2n)^{\frac{8n^{2}+2n-1}{4n}}\abs{\S^{4n-1}}^{\frac{1}{4n}}\abs{\S^{2n}}^{\frac{4n-1}{2n}} \brac{1 + (d-1)^\frac{4n-1}{4n}}\text{,}
	\end{equation}
	then $\#_{1,4n-1}d$ is attained.
\end{proposition}
\begin{proof}
	For the sake of conciseness, we define \( c_{n} = (2n)^{\frac{8n^{2}+2n-1}{4n}}\abs{\S^{4n-1}}^{\frac{1}{4n}}\abs{\S^{2n}}^{\frac{4n-1}{2n}} \).
	Assume on the contrary that $\#_{1,4n-1}d$ is not attained. 
	Then, by the energy identity \Cref{th:energy-id-general}, we obtain that there would exist $d_1,\dotsc, d_N \in \Z\setminus \{d,0\}$ such that $d = \sum_{i=1}^N d_i$ and
	\begin{equation}\label{eq:rerereren}
		\#_{1,4n-1} d = \sum_{i=1}^N \#_{1,4n-1} d_i \ge c_{n} \sum_{i=1}^N |d_i|^{\frac{4n-1}{4n}}\text{,}
	\end{equation}
	where the last inequality is a consequence of \Cref{pr:reverseenergyestimaten}. 
	
	Since \( \frac{4n-1}{4n} < 1 \), by the subadditivity of \( t \mapsto t^{\frac{4n-1}{4n}} \),
	\[
	\sum_{i=1}^N |d_i|^\frac{4n-1}{4n} \ge |d_1|^{\frac{4n-1}{4n}} + \brac{\sum_{i=2}^N |d_i|}^\frac{4n-1}{4n} = |d_1|^{\frac{4n-1}{4n}} + \brac{\sum_{i=1}^N |d_i| - |d_1|}^\frac{4n-1}{4n}\text{.}
	\]
	Let $S= \sum_{i=1}^N |d_i|$. Since, the function $t \mapsto t^{\frac{4n-1}{4n}} + \brac{S - t}^\frac{4n-1}{4n}$ is concave on $\left[1, S -1\right]$ and attains its maximum at the midpoint of the interval, we get that the minimum on that interval is attained at both endpoints. 
	Now, since $d = \sum_{i=1}^N d_i$, we have that $d \le \sum_{i=1}^N |d_i| = S$.
	Since $1\le |d_1| \le S - 1$, it follows that 
	\begin{equation}\label{eq:tralalalalan}
		\sum_{i=1}^N |d_i|^\frac{4n-1}{4n} \ge 1 + \brac{\sum_{i=1}^N |d_i| -1}^\frac{4n-1}{4n} \ge 1+(d-1)^\frac{4n-1}{4n}\text{.}
	\end{equation}
	Combining \eqref{eq:rerereren} with \eqref{eq:tralalalalan} we get
	\[
	\#_{1,4n-1} d \ge c_{n}\brac{1+(d-1)^\frac{4n-1}{4n}}\text{,}
	\]
	a contradiction with \eqref{eq:energyboundforexistencen}.
\end{proof}

\subsection{Proof of \Cref{pr:hopf234}}
Having Proposition~\ref{pr:energyboundforexistencen} at hand, in order to prove \Cref{pr:hopf234}, it suffices to construct any map in $W^{1,3}(\S^3,\S^2)$ with sufficiently small energy. The explicit maps in \Cref{thm:higherdegreemaps} were produced by ChatGPT and we postpone the proof of their properties to \Cref{a:aiexplicit}.


\begin{theorem}\label{thm:higherdegreemaps}
For each $d\in\{2,3,4,5,6\}$, the map $u_d$ defined below is Lipschitz,
belongs to $W^{1,3}(\S^3,\S^2)$ and satisfies
\[
 H(u_d)=d,\qquad
 E_{1,3}(u_d)<32\sqrt2\,\pi^2\bigl(1+(d-1)^{3/4}\bigr)\text{.}
\]
More precisely, the following bounds hold:
\begin{center}
\begin{tabular}{c|l}
$d$ & $E_{1,3}(u_d)/(32\sqrt2\,\pi^2)$ \\
\hline\hline
2 & $<2$ \\
\hline
3 & $<8/3$ \\
\hline
4 & $=625\pi/(432\sqrt2)<13/4$ \\
\hline
5 & $<191/50$ \\
\hline
6 & $<13/3$. \\
\end{tabular}
\end{center}

We denote by \( H \) the Hopf map defined in \eqref{eq:Hopf}.
For each degree, we specify a map $w_d\colon\S^3\to\S^3$ and set
\begin{equation}\label{eq:ud}
 u_d\coloneqq H\circ w_d\text{.}
\end{equation}

For degree $2$, let
\begin{equation}\label{eq:w2}
 w_2(z_1,z_2)=
 \begin{cases}
 \displaystyle
 \frac{(2z_1,\,z_2^2/|z_2|)}
 {\sqrt{4|z_1|^2+|z_2|^2}}\text{,}&z_2\ne0\text{,}\\[8pt]
 (z_1,0)\text{,}&z_2=0\text{.}
 \end{cases}
\end{equation}
For degree $3$, let
\begin{equation}\label{eq:w3}
 w_3(z_1,z_2)=
 \frac{(2z_1,z_2^3)}{\sqrt{4|z_1|^2+|z_2|^6}}\text{.}
\end{equation}
For degree $4$, let
\begin{equation}\label{eq:w4}
 w_4(z_1,z_2)=
 \begin{cases}
 \displaystyle
 \frac{\bigl(z_1^2/|z_1|^{1/2},\,z_2^2/|z_2|^{1/2}\bigr)}
 {\sqrt{|z_1|^3+|z_2|^3}}\text{,}&z_1\ne0\text{ and }z_2\ne0\text{,}\\[8pt]
 (z_1^2,0)\text{,}&z_2=0\text{,}\\[2pt]
 (0,z_2^2)\text{,}&z_1=0\text{.}
 \end{cases}
\end{equation}
The last two cases are disjoint on $\S^3$, because $|z_1|^2+|z_2|^2=1$.

For degrees $5$ and $6$, use the parameters
\begin{equation}\label{eq:parameters}
\begin{array}{c|cccc}
d&m&k&c&\lambda\\ \hline
5&2&3&4/5&15/8\\[2pt]
6&3&3&1&2
\end{array}
\end{equation}
and define, at every point of $\S^3$,
\begin{equation}\label{eq:w56}
 w_d(z_1,z_2)=
 \frac{\bigl(|z_1|z_1^m+c|z_2|z_2^k,\,\lambda z_1z_2\bigr)}
 {\sqrt{\bigl||z_1|z_1^m+c|z_2|z_2^k\bigr|^2
             +\lambda^2|z_1|^2|z_2|^2}}\text{.}
\end{equation}
In particular, $w_d(z_1,0)=(z_1^m,0)$ and
$w_d(0,z_2)=(z_2^k,0)$ for $d=5,6$.
\end{theorem}

\begin{proof}[Proof of \Cref{pr:hopf234}]
The case $d=0$ is immediate, since constant maps have 0 energy. The case $d=1$ follows from \cite[Theorem 1.4]{Wang-Zhang}.

For $d\in\{2,3,4,5,6\}$, Theorem~\ref{thm:higherdegreemaps} gives
\[
\#_{1,3}d
\le E_{1,3}(u_d)
<32\sqrt{2}\,\pi^2\bigl(1+(d-1)^{3/4}\bigr)\text{.}
\]
Since
\[
2^{9/4}|\S^3|^{1/4}|\S^2|^{3/2}
=2^{9/4}(2\pi^2)^{1/4}(4\pi)^{3/2}
=32\sqrt{2}\,\pi^2\text{,}
\]
Proposition~\ref{pr:energyboundforexistencen}, applied with $n=1$,
implies that $\#_{1,3}d$ is attained.

Finally, composition with an orientation-reversing isometry
$R\colon \S^3\to\S^3$ preserves the energy and reverses the Hopf degree:
\[
E_{1,3}(u\circ R)=E_{1,3}(u)\text{,}
\qquad
\deg_H(u\circ R)=-\deg_H(u)\text{.}
\]
This gives an energy-preserving bijection between the classes of
Hopf degree $d$ and $-d$, whence existence for positive degrees implies
existence for the corresponding negative degrees.
\end{proof}
\subsection{Proof of \Cref{pr:f(d)behavior}}
In view of \Cref{pr:reverseenergyestimaten}, one might expect that
\[
    \#_{1,3}d=|d|^{3/4}\#_{1,3}1\text{.}
\]
However, this identity fails for every $d\ge2$. Indeed, setting
\[
    f(d)\coloneqq\frac{\#_{1,3}d}{\#_{1,3}1}\text{,}
\]
we have
\begin{equation}\label{eq:strict-energy-bounds}
    d^{3/4}<f(d)<d
    \qquad\text{for every }d\ge2\text{,}
\end{equation}
as asserted in \Cref{pr:f(d)behavior}.

\begin{proof}[Proof of \Cref{pr:f(d)behavior}]
Recall that
\[
    \#_{1,3}1=32\sqrt{2}\,\pi^2\text{.}
\]
We first prove the upper bound $f(d)<d$.
By \Cref{thm:higherdegreemaps},
\[
    \#_{1,3}2<2\#_{1,3}1\text{.}
\]
Moreover, the minimal energy is subadditive:
by \cite[Theorem~1.2]{VS20},
\[
    \#_{1,3}(a+b)\le\#_{1,3}a+\#_{1,3}b
    \qquad\text{for all }a,b\in\Z\text{.}
\]
For any integer $d\ge2$, write $d=2k+r$, where $k\ge1$ and
$r\in\{0,1\}$. Since $\#_{1,3}0=0$, it follows that
\[
    \#_{1,3}d
    \le k\#_{1,3}2+r\#_{1,3}1
    <(2k+r)\#_{1,3}1
    =d\#_{1,3}1\text{.}
\]

We now prove the strict lower bound $f(d)>d^{3/4}$.
By \Cref{pr:reverseenergyestimaten}, we already know that
\[
    \#_{1,3}d\ge d^{3/4}\#_{1,3}1\text{.}
\]
Assume, on the contrary, that equality holds for some
$d\ge2$. Since
\[
    d^{3/4}<1+(d-1)^{3/4}\text{,}
\]
we obtain
\[
    \#_{1,3}d
    <32\sqrt{2}\,\pi^2\bigl(1+(d-1)^{3/4}\bigr)\text{.}
\]
Thus, \Cref{pr:energyboundforexistencen} yields a minimizer $u_d\in W^{1,3}(\S^3,\S^2)$. By standard regularity theory for minimizing harmonic maps \cite[Corollary~3.2]{HLp},
$u_d\in C^{1,\beta}(\S^3,\S^2)$ for some $\beta\in(0,1)$.
Moreover,
\begin{equation}\label{one}
    \deg_H(u_d)=d
    \qquad
    \text{and}
    \qquad
    E_{1,3}(u_d)
    =\#_{1,3}d
    =32\sqrt{2}\,\pi^2d^{3/4}\text{.}
\end{equation}
In particular, $u_d$ attains the equality in the degree
estimate given by Proposition~\ref{pr:reverseenergyestimaten}. Equality in the pointwise (a.e.)\ inequality used in its
proof, together with continuity of $\nabla u_d$, gives
\begin{equation}\label{two}
    |\nabla u_d|^2=2|u_d^*\omega_{\S^2}|
    \qquad\text{on }\S^3\text{.}
\end{equation}

Choose a $1$-form $\alpha$ on $\S^3$ satisfying
$\dext\alpha=u_d^*\omega_{\S^2}$. By the definition of the
Hopf degree \eqref{eq:Hopfdegree},
\[
    \int_{\S^3}\alpha\wedge\dext\alpha=16\pi^2d\text{.}
\]
Combining this identity with \eqref{one} and \eqref{two},
we obtain
\[
    J(\alpha)
    \coloneqq
    \frac{\bigl(\int_{\S^3}|\dext\alpha|^{3/2}\bigr)^{4/3}}
         {\int_{\S^3}\alpha\wedge\dext\alpha}
    =
    \frac{(16\pi^2d^{3/4})^{4/3}}{16\pi^2d}
    =2|\S^3|^{1/3}\text{.}
\]
Thus, $\alpha$ attains the minimum in
\cite[Theorem~1.1]{Wang-Zhang}, and its equality
characterization yields
\[
    \alpha=\Phi^*\xi+\gamma\text{,}
    \qquad \dext\gamma=0\text{,}
\]
where $\Phi$ is an orientation-preserving conformal diffeomorphisms and $\xi$ is a positive Killing $1$-form on $\S^3$.

By \cite[Proposition~2.3(2)]{Wang-Zhang}, $|\xi|$ is constant.
This constant is strictly positive: otherwise $\xi\equiv0$,
hence $\dext\alpha\equiv0$, contradicting $\deg_H(u_d)=d>0$.
By the definition of a positive Killing form,
\[
    \ast\dext\xi=\operatorname{curl}\xi=2\xi\text{.}
\]
Consequently,
\[
    |\dext\xi|=2|\xi|>0
    \qquad\text{everywhere on }\S^3\text{.}
\]
Since
\[
    \dext\alpha=\Phi^*(\dext\xi)
\]
and $\Phi$ is a diffeomorphism, $\dext\alpha$ is nowhere
vanishing. Therefore, by \eqref{two},
\[
    |\nabla u_d|^2=2|\dext\alpha|>0
    \qquad\text{everywhere on }\S^3\text{.}
\]
The Euler--Lagrange system for $u_d$ is therefore uniformly elliptic
 and elliptic bootstrapping yields
$u_d\in C^\infty(\S^3,\S^2)$.

By conformal invariance, we may now apply
\cite[Theorem~B.3]{Wang-Zhang} with $\theta=0$.
Taking into account the normalization
$E_{\mathrm{Ni}}=E_{1,3}/\sqrt{2}$, its equality
classification yields
\[
    u_d=R\circ H\circ\Phi\text{,}
\]
where $R\colon\S^2\to\S^2$ is an isometry, $H\colon \S^3\to \S^2$ is the Hopf map defined in \eqref{eq:Hopf}, and
$\Phi$ is a conformal diffeomorphism.
Hence, using Propositions~\ref{prop:HopfDegLeftComp} and~\ref{prop:HopfDegRightComp},
\[
    |\deg_H(u_d)|
    =(\deg R)^2|\deg\Phi|
    =1\text{,}
\]
contradicting $\deg_H(u_d)=d\ge2$.
\end{proof}

\subsection{Proof of \Cref{pr:existenceindeg1n=2}}
\begin{proof}[Proof of \Cref{pr:existenceindeg1n=2}]
We let \( H_{q} \colon \S^{7} \to \S^{4} \) be the quaternionic Hopf map defined in \eqref{eq:quaternionicHopf}.
Using \eqref{eq:energyqhopf}, we compute
\begin{equation}\label{eq:C2value}
C_2 = 4^7|\S^7|^{1/7}|\S^4|^{7/4} = \frac{4^{7}\pi^4}{3^{15/8}} = \frac{E_{1,7}(H_q)}{3^{7/8}}\text{,}
\end{equation}
where $C_2$ is the value from \Cref{pr:reverseenergyestimaten}.

 Assume on the contrary that there is no map $u\colon \S^7\to \S^4$ minimizing $E_{1,7}$ among maps of degree \(1\). Then, from the energy identitity \Cref{th:energy-id-general} there would exist $N$ and $d_i\in \Z\setminus \{0,\pm1\}$ for $i\in\{1,\dotsc, N\}$ such that
\begin{equation}\label{eq:conditionondi}
 \sum_{i=1}^N d_i = 1 
\end{equation}
and 
\[
\#_{1,7}1 = \sum_{i=1}^N \#_{1,7}d_i\text{.}
\]

By \Cref{pr:reverseenergyestimaten}, we obtain
\[
 \#_{1,7}1 = \sum_{i=1}^N \#_{1,7}d_i \ge C_2 \sum_{i=1}^N |d_i|^{\frac{7}{8}}\text{.}
\]
Relabelling the indices of $d_i$ in such a way that $d_1,\dotsc,d_K >0$ and $d_{K+1},\dotsc, d_N<0$, we obtain that $\sum_{i=1}^K |d_i|\ge 3$ and  $\sum_{i=K+1}^N |d_i|\ge 2$. Thus, using the subadditivity of $t\mapsto t^{\frac78}$, we have
\[
\begin{split}
 \#_{1,7}1 = \sum_{i=1}^N \#_{1,7}d_i \ge C_2 \sum_{i=1}^N |d_i|^{\frac{7}{8}} &\ge C_2 \brac{\brac{\sum_{i=1}^K |d_i|}^{\frac{7}{8}} + \brac{\sum_{i=K+1}^N |d_i|}^{\frac{7}{8}} } \\ 
 &\ge C_2 \brac{3^{7/8} + 2^{7/8}}\text{.}
\end{split}
\]
Therefore, using \eqref{eq:C2value}, we would obtain
\[
 E_{1,7}(H_2) \ge \#_{1,7}1 \ge \frac{E_{1,7}(H_2)}{3^{7/8}}\brac{3^{7/8} + 2^{7/8}}\text{.}
\]
This is a contradiction, and concludes the proof.
\end{proof}

\appendix
\section{Explicit constructions of mappings}

In this appendix, we present the explicit construction of some mappings that we instrumental to our arguments.
We first explain the construction of mappings of spheres that have Brouwer degree equal to \( 1 \) and that are constant outside of a small ball, and then we present the AI generated maps featured in Theorem~\ref{thm:higherdegreemaps}.

\subsection{Maps of spheres constant outside a small set}\label{a:explicit}
\begin{lemma}\label{a1-3-lmm:top-deg1}
Let \( x_0, b \in \mathbb{S}^m \) and \( 0 < r < 1 \). There exists a smooth map \( f_{x_0,r} \colon \mathbb{S}^m \to \mathbb{S}^m \) such that
\[
\deg(f_{x_0,r}) = 1\text{,} \qquad |\nabla^k f_{x_0,r}| \leq C r^{-k}\text{,} \qquad f_{x_0,r} \equiv b \text{ outside } B(x_0, r)\text{,}
\]
for some constant \( C > 0 \) independent of \( r \).
\end{lemma}

\begin{proof}
Let \( x_N = (0,\ldots,1) \) and \( x_S = (0,\ldots,-1) \) be the north and south poles of \( \mathbb{S}^m \), and let \( \Pi \colon \mathbb{S}^m \setminus \{x_N\} \to \mathbb{R}^m \) denote the stereographic projection, with inverse
\[
\Pi^{-1}(y) = \left( \frac{2y}{1 + |y|^2}, \frac{-1 + |y|^2}{1 + |y|^2} \right)\text{.}
\]
By rotating the domain and codomain, we may assume that \( x_0 = x_S \) and \( b = x_N \).

\textit{Step 1.} Define the dilation \( \mu_r(y) = y/r \), and consider the map
\[
\tilde{f} = \Pi^{-1} \circ \mu_r \circ \Pi\text{.}
\]
Then, \( \tilde{f} \) maps \( B(x_S, r) \) onto the southern hemisphere and satisfies \( |\nabla^k \tilde{f}| \leq Cr^{-k} \) on \( B(x_S, r) \) by scaling, since \( \Pi \) and \( \Pi^{-1} \) are smooth diffeomorphisms on this region.

\textit{Step 2.} Let \( g \colon \mathbb{S}^m \to \mathbb{S}^m \) be the smooth map
\[
g(x) = (-2x_1x_{m+1}, \ldots, -2x_mx_{m+1}, 1 - 2x_{m+1}^2)\text{,}
\]
which maps the southern hemisphere onto \( \mathbb{S}^m \) and satisfies \( g(x) = x_N \) on the equator.

Define
\[
f_{x_S, r}(x) =
\begin{cases}
g(\tilde{f}(x)) & \text{if } x \in B(x_S, r)\text{,} \\
x_N & \text{otherwise.}
\end{cases}
\]
Then, \( f_{x_S, r} \) has degree 1 and satisfies the desired bounds. Composing with rotations yields the general case \( f_{x_0,r} \), completing the proof.
\end{proof}

\subsection{Explicit maps of higher order generated by AI}\label{a:aiexplicit}
\begin{proof}[Proof of Theorem~\ref{thm:higherdegreemaps}]
\textsc{Step 1: Well-definedness and regularity}
 
 Each denominator in \eqref{eq:w2}--\eqref{eq:w4} is positive on
$\S^3$, and the squared norm of its numerator is the square of that
denominator. For \eqref{eq:w56}, simultaneous vanishing of both numerator
components would imply $z_1z_2=0$. If $z_2=0$, the first component is
$z_1^m\ne0$; if $z_1=0$, it is $c z_2^k\ne0$. Thus this numerator never
vanishes on $\S^3$ either. Its norm has a strictly positive minimum by
compactness. All the maps $w_d$ therefore take values in $\S^3$.

We record the elementary regularity observation needed for the formulas
with noninteger radial powers. For $a\ge1$ and a positive integer $m$,
define a complex-valued map on the closed unit disc by
\[
 F(0)=0,\qquad F(r\e^{\imath\theta})=r^a \e^{\imath m\theta}\quad(0<r\le1).
\]
This is single-valued because $m$ is an integer. It is continuous at zero.
In the orthonormal polar directions its derivatives have norms
$a r^{a-1}$ (for the radial derivative) and $m r^{a-1}$ (for the angular derivative), and hence the operator norm of the differential is at most
$\max\{a,m\}$. Integrating along a line segment proves the same Lipschitz
bound when that segment avoids zero. If it passes through zero, split it
there and use continuity. Consequently $F$ is Lipschitz on the disc.
If $a>1$, these derivatives tend to zero at the origin, and
$|F(z)|/|z|=|z|^{a-1}\to0$; thus, $F$ is $C^1$, with derivative zero at
the origin.

Apply this observation with $(a,m)=(1,2)$ to the second numerator
component of $w_2$, and with $(a,m)=(3/2,2)$ to both components of $w_4$.
For \eqref{eq:w56}, the factors $|z|z^m$ have $(a,m)=(m+1,m)$ and are
$C^1$. The lift $w_3$ is smooth. Normalization of a vector is smooth away
from zero; all the normalizing factors have a positive lower bound on
the compact domain. It follows that $w_2$ is Lipschitz and that
$w_3,w_4,w_5,w_6$ are $C^1$ (with $w_3$ smooth). Since $H$ is smooth,
all five maps $u_d$ are Lipschitz. In particular their weak derivatives
are bounded, so they belong to $W^{1,3}(\S^3,\S^2)$.

\textsc{Step 2: Hopf coordinates and integration formulas}

Define
\begin{equation}\label{eq:Phi}
 \Phi(t,\alpha,\beta)
  =(\cos t\,\e^{\imath\alpha}\text{,}\sin t\,\e^{\imath\beta}),\qquad
 (t,\alpha,\beta)\in(0,\pi/2)\times(0,2\pi)\times(0,2\pi)\text{.}
\end{equation}
Its image is the open set
\[
 U=\{(z_1,z_2)\in\S^3 \colon
       z_1\notin[0,\infty),\ z_2\notin[0,\infty)\}\text{,}
\]
where $[0,\infty)$ is the nonnegative real axis in $\C$, including zero.
Indeed, the two radii in \eqref{eq:Phi} are positive, and neither angle
represents the positive real direction. Conversely, for a point of $U$
the unique inverse is
\[
 t=\arctan\frac{|z_2|}{|z_1|}\text{,}\qquad
 \alpha=\arg{z_1}\text{,}\qquad
 \beta=\arg{z_2}\text{.}
\]
Here $\arg{z}$ is the unique $\theta\in(0,2\pi)$
such that $z=|z|\e^{\imath\theta}$. It is a smooth real-valued function on
$\C\setminus[0,\infty)$: explicitly,
\[
 {\arg}(x+\imath y)
 =\pi-2\arctan\frac{y}{\sqrt{x^2+y^2}-x}\text{.}
\]
The denominator is positive on this domain. 
Thus, $\Phi$ is a smooth
diffeomorphism onto $U$.

Its coordinate tangent vectors are
\[
\begin{aligned}
 \Phi_t&=(-\sin t\,\e^{\imath\alpha},\cos t\,\e^{\imath\beta})\text{,}\\
 \Phi_\alpha&=(\imath\cos t\,\e^{\imath\alpha},0)\text{,}\\
 \Phi_\beta&=(0,\imath\sin t\,\e^{\imath\beta})\text{.}
\end{aligned}
\]
They are pairwise orthogonal, with norms $1$, $\cos t$, $\sin t$.
Consequently,
\begin{equation}\label{eq:metric}
 g=\mathrm dt^2+\cos^2t\,\mathrm d\alpha^2
                  +\sin^2t\,\mathrm d\beta^2\text{,}
 \qquad \sqrt{\det g}=\sin t\cos t>0\text{.}
\end{equation}
The surface change of variables theorem applies because the
parametrization is injective, smooth, and has rank \(3\) everywhere.
The complement of $U$ is contained in the union of the two smooth
two-dimensional spheres
$\S^3\cap\{\operatorname{Im}z_1=0\}$ and
$\S^3\cap\{\operatorname{Im}z_2=0\}$; it has zero three-dimensional
volume. Therefore, for every nonnegative measurable function $F$, and
also for every integrable signed function,
\begin{equation}\label{eq:integration}
 \int_{\S^3}F
 =\int_0^{\pi/2}\int_0^{2\pi}\int_0^{2\pi}
 F(\Phi(t,\alpha,\beta))\sin t\cos t
 \dd\beta\dd\alpha\dd t\text{.}
\end{equation}

If a map $w$ has target coordinates $(T,A,B)$, its signed Jacobian is
\begin{equation}\label{eq:Jac}
 \Jac w(\Phi)
 =\frac{\sin T\cos T}{\sin t\cos t}
      \det\frac{\partial(T,A,B)}{\partial(t,\alpha,\beta)}\text{.}
\end{equation}
This follows by taking the ratio of the two volume forms or by
the ordinary real determinant identity
\[
 \Jac w(\Phi)=
 \frac{\det(w\circ\Phi,\partial_t(w\circ\Phi),
                    \partial_\alpha(w\circ\Phi),
                    \partial_\beta(w\circ\Phi))}
      {\det(\Phi,\Phi_t,\Phi_\alpha,\Phi_\beta)}\text{.}
\]

\textsc{Step 3: Computation of the degree}

Recall that $u_d=H\circ w_d$, where $H\colon \S^3\to \S^2$ is the
Hopf map, which has Hopf invariant $1$. 
By Proposition~\ref{prop:HopfDegRightComp},
\[
\deg_H(u_d)
=\deg_H(H\circ w_d)
=\deg(w_d)\deg_H(H)
=\deg(w_d)\text{.}
\]
Thus, it suffices to determine the degree of the lift $w_d \colon \S^3\to \S^3$.

For $d=2,3,4$, this is particularly transparent in Hopf coordinates
\[
(z_1,z_2)=(\cos t\,\e^{\imath\alpha},\sin t\,\e^{\imath\beta})\text{,}
\qquad 0<t<\frac{\pi}{2}\text{.}
\]
In these coordinates, the lifts take the form
\[
\begin{aligned}
w_2\colon\quad (t,\alpha,\beta)
&\longmapsto
\left(\arctan\!\left(\frac{\tan t}{2}\right),\alpha,2\beta\right)\text{,}
\\[4pt]
w_3\colon\quad (t,\alpha,\beta)
&\longmapsto
\left(\arctan\!\left(\frac{\sin^3 t}{2\cos t}\right),\alpha,3\beta\right)\text{,}
\\[4pt]
w_4\colon\quad (t,\alpha,\beta)
&\longmapsto
\left(\arctan\!\left((\tan t)^{3/2}\right),2\alpha,2\beta\right)\text{,}
\end{aligned}
\]
where the angular coordinates are understood modulo $2\pi$.

We denote for future purposes 
\begin{equation}\label{eq:torus-data}
\begin{array}{c|cc|c}
d&m&k&T(t)\\ \hline
2&1&2&\arctan(\tan t/2)\\[2pt]
3&1&3&\arctan(\sin^3t/(2\cos t))\\[2pt]
4&2&2&\arctan((\tan t)^{3/2})\text{.}
\end{array}
\end{equation}

In each case, the first component is an increasing diffeomorphism
of $(0,\pi/2)$ onto itself. Consequently, it contributes neither
additional preimages nor an orientation reversal. The number of
preimages is determined entirely by the angular components:
$w_2$ doubles the second angle, $w_3$ triples it, and $w_4$
doubles both angles. Every target point with both complex coordinates
nonzero therefore has, respectively, $2$, $3$, or $4$ preimages,
all with positive orientation. Hence
\[
\deg_H(u_2)=2\text{,}\qquad
\deg_H(u_3)=3\text{,}\qquad
\deg_H(u_4)=4\text{.}
\]

This counting argument can also be verified directly using the
integral formula for the degree. For a lift of the form
\[
(t,\alpha,\beta)\longmapsto(T(t),m\alpha,k\beta)\text{,}
\]
with $T$ an increasing diffeomorphism of $(0,\pi/2)$ onto itself,
that formula gives
\[
\begin{aligned}
\deg(w)
&=\frac{1}{2\pi^2}\int_{S^3}\operatorname{Jac}w\\
&=2mk\int_0^{\pi/2}T'(t)\sin T(t)\cos T(t)\,\mathrm{d}t\\
&=2mk\int_0^{\pi/2}\sin s\cos s\,\mathrm{d}s\\
&=mk\text{.}
\end{aligned}
\]

For the maps with $d=5,6$, the angular variables interact, so their
coordinate expressions do not have this separated form. Nevertheless,
their degrees admit an equally short counting explanation. From
their definitions,
\[
w_d(\e^{\imath\alpha},0)=(\e^{\imath m\alpha},0)\text{,}
\qquad
w_d(0,\e^{\imath\beta})=(\e^{\imath k\beta},0)\text{,}
\]
where $(m,k)=(2,3)$ for $d=5$ and $(m,k)=(3,3)$ for $d=6$.
Moreover, the second component of $w_d$ vanishes precisely when
$z_1z_2=0$. Thus $(1,0)$ has exactly $m+k$ preimages:
$m$ from the first displayed restriction and $k$ from the second.

Each of these preimages has local degree $+1$: the angular
derivative is positive, and the differential in the transverse
complex direction is the multiplication by a nonzero complex number,
which preserves orientation. Therefore,
\[
\deg_H(u_5)=\deg(w_5)=2+3=5\text{,}
\qquad
\deg_H(u_6)=\deg(w_6)=3+3=6\text{.}
\]

\textsc{Step 4: Energy bounds for degrees 2, 3, and 4}

For the profiles in \eqref{eq:torus-data},
\[
 u_d\circ\Phi
 =(\sin(2T(t))\e^{\imath(m\alpha-k\beta)},\cos(2T(t)))\text{.}
\]
Differentiating this expression and using the inverse metric in
\eqref{eq:metric} gives
\begin{equation}\label{eq:torus-energy}
 |\nabla u_d|^2
 =4T'(t)^2+\sin^2(2T(t))
      \left(\frac{m^2}{\cos^2t}+\frac{k^2}{\sin^2t}\right)\text{.}
\end{equation}
In particular, this expression is independent of the angular variables,
and its $3/2$ power is integrated against
$4\pi^2\sin t\cos t\dd t$.

\textbf{Degree 2:} 
For $T(t)=\arctan\bigl(\frac{\tan t}{2}\bigr)$,
\[
 T'(t)=\frac{2}{4\cos^2t+\sin^2t}\text{,}\qquad
 \sin(2T(t))=\frac{4\sin t\cos t}{4\cos^2t+\sin^2t}\text{.}
\]
Substitution in \eqref{eq:torus-energy} gives
\[
 |\nabla u_2|^2
 =\frac{5-3\sin^2t}{(1-\tfrac34\sin^2t)^2}\text{.}
\]
Setting $s=\sin^2t$, so that $\mathrm ds=2\sin t\cos t\,\mathrm dt$, yields
\begin{equation}\label{eq:E2}
 E_{1,3}(u_2)=2\pi^2\int_0^1
       \frac{(5-3s)^{3/2}}{(1-\tfrac34s)^3}\dd s\text{.}
\end{equation}
By convexity of $x\mapsto x^{3/2}$,
\[
 (5-3s)^{3/2} = ((1-s)5 + s2)^{3/2}\le(1-s)5\sqrt5+s\,2\sqrt2\text{.}
\]
Since $5\sqrt5<8\sqrt2$ (equivalently, $125<128$), for $0\le s<1$,
\[
 (5-3s)^{3/2}<8\sqrt2(1-s)+2\sqrt2s
             =8\sqrt2(1-\tfrac34s)\text{.}
\]
Therefore,
\[
 E_{1,3}(u_2)<16\sqrt2\,\pi^2\int_0^1(1-\tfrac34s)^{-2}\dd s
          =64\sqrt2\,\pi^2=2E_{1,3}(H)\text{.}
\]

\textbf{Degree 3:}
For $T(t)=\arctan(\sin^3t/(2\cos t))$, formula
\eqref{eq:torus-energy} with $s=\sin^2t$ gives
\begin{equation}\label{eq:E3density}
 |\nabla u_3|^2
 =\frac{32s^2(9-10s+2s^2)}{(4-4s+s^3)^2}\text{,}\qquad
 E_{1,3}(u_3)
 =2\pi^2(32)^{3/2}\int_0^1\frac{s^3(9-10s+2s^2)^{3/2}}
                     {(4-4s+s^3)^3}\dd s\text{.}
\end{equation}
We have $2\pi^2(32)^{3/2} = 8(32\sqrt{2}\pi^2) = 8 E_{1,3}(H)$. 
On $[0,1]$, the quantities $3-2s$, $9-10s+2s^2$, and $4-4s+s^3$ are
positive. Additionally, observe that
\begin{equation}\label{eq:sqrtmajorant}
 \sqrt{9-10s+2s^2}
 \le\frac{9-11s+3s^2}{3-2s}
\end{equation}
(to see that, observe that \( 3-2s = \sqrt{9-12s+4s^{2}} \) and apply the Arithmetic Mean--Geometric Mean inequality),
which implies that
\begin{equation}
    8\frac{s^3(9-10s+2s^2)^{3/2}}{(4-4s+s^3)^3} \leq 8\frac{s^3(9-10s+2s^2)}{(4-4s+s^3)^3}\frac{9-11s+3s^2}{3-2s}\text{.}
\end{equation}
Define the following explicitly specified antiderivative majorant:
\[
 F(s)=\frac{s^4(46-52s+14s^2)}{3(4-4s+s^3)^2}\text{.}
\]
A direct computation gives
\begin{align}\label{eq:Fcertificate}
 &F'(s)-
 \frac{8s^3(9-10s+2s^2)(9-11s+3s^2)}
      {(3-2s)(4-4s+s^3)^3}\notag\\
 &\quad=
 \frac{4s^3(1-s)}{3(3-2s)(4-4s+s^3)^3}
 \left[21\left(s-\frac9{14}\right)^2+\frac9{28}
                 +31(1-s)^3+26(1-s)^4\right]\text{.}
\end{align}
The right-hand side is nonnegative on $[0,1]$ and positive on $(0,1)$.
Combining \eqref{eq:E3density}--\eqref{eq:Fcertificate} proves that
\[
 \frac{E_{1,3}(u_3)}{E_{1,3}(H)}<\int_0^1F'(s)\dd s=F(1)-F(0)=\frac83\text{.}
\]
Finally, $(5/3)^4=625/81<8$, so that
$8/3<1+2^{3/4}$, as required.

\textbf{Degree 4:}
Let $r=\tan t$. Here $T=\arctan(r^{3/2})$, and
\eqref{eq:torus-energy} simplifies to
\[
 |\nabla u_4|^2
 =\frac{25r(1+r^2)^2}{(1+r^3)^2}\text{,}\qquad
 \sin t\cos t\dd t=\frac{r}{(1+r^2)^2}\dd r\text{.}
\]
Consequently
\begin{align*}
 E_{1,3}(u_4)
 &=500\pi^2\int_0^\infty
       \frac{r^{5/2}(1+r^2)}{(1+r^3)^3}\dd r\\
 &=1000\pi^2\int_0^\infty\frac{r^{5/2}}{(1+r^3)^3}\dd r
 =\frac{1000\pi^2}{3}\,B\left(\frac76,\frac{11}6\right)\text{.}
\end{align*}
The two terms in the first integral are equal by $r\mapsto1/r$;
the last substitution is $y=r^3$.
Here $B(a,b){\coloneqq}\int_0^\infty y^{a-1}(1+y)^{-a-b}\dd y$ is the beta function.
The beta identity $B(a,b)=\Gamma(a)\Gamma(b)/\Gamma(a+b)$, together
with the recurrence and reflection formulas for
$\Gamma(a){\coloneqq}\int_0^\infty y^{a-1}\e^{-y}\dd y$ (see, e.g., \cite[Section 5]{NIST}), gives
\[
 B\left(\frac76,\frac{11}6\right)
 =\frac5{72}\Gamma\left(\frac16\right)
                  \Gamma\left(\frac56\right)=\frac{5\pi}{36}\text{.}
\]
Thus,
\begin{equation}\label{eq:E4}
 E_{1,3}(u_4)=\frac{1250}{27}\pi^3,
 \qquad \frac{E_{1,3}(u_4)}{E_{1,3}(H)}=\frac{625\pi}{432\sqrt2}.
\end{equation}
For an entirely rational comparison, use $\pi<22/7$ and $\sqrt2>7/5$:
\[
 \frac{625\pi}{432\sqrt2}<\frac{68750}{21168}<\frac{13}{4}.
\]
Moreover, $(9/4)^4=6561/256<27$, whence
$13/4<1+3^{3/4}$.

\textsc{Step 5: Degrees 5 and 6}

Recall that $(m,k,c,\lambda)$ is one of the two rows in
\eqref{eq:parameters}. We first derive an algebraic upper bound for the
energy integral, and then certify that bound exactly.

\textbf{The differential and angular averaging.}
Write the two unnormalized components in \eqref{eq:w56} as
\[
 f=|z_1|z_1^m+c|z_2|z_2^k\text{,}\qquad g=\lambda z_1z_2\text{.}
\]
Thus, $u_{d}=(2f\overline g,|f|^2-|g|^2)/(|f|^2+|g|^2)$.
The computation of the differential of \( u_{d} \) is easier through the stereographic projection.
When \( f \neq 0 \), the stereographic projection of \( u_{d} \) is given by \( g/f \), and its differential reads
\[
	\frac{f\dif g - g\dif f}{f^{2}}\text{.}
\]
Since the metric element for the stereographic projection at a point
\(
	(X,Y) = \bigl(\frac{x}{1+z},\frac{y}{1+z}\bigr)
\)
is given by \( \frac{4}{(1+X^{2}+Y^{2})^{2}}(\dif X^{2} + \dif Y^{2}) \), we obtain
\begin{equation}\label{eq:component-differential}
 |\dext u|^2
 =\frac{4|g\dext f-f\dext g|^2}
             {(|f|^2+|g|^2)^2}\text{.}
\end{equation}
The same identity holds where $f=0$ by using the coordinate $f/g$, or
by continuity. At no point do both components vanish.

Set $s=\sin^2t$ and $\theta=m\alpha-k\beta$. Define the following real
polynomials in $s$:
\begin{equation}\label{eq:poly-data}
\begin{split}
 a&=(1-s)^{m+1}\text{,}\\
 b&=c^2s^{k+1}\text{,}\\
 D&=a+b+\lambda^2s(1-s)\text{,}\\
 e&=4ab\text{,}\\
 L&=a\left[(1+(m-1)s)^2+(m-1)^2s+1-s\right]\\
  &\quad+b\left[(k-(k-1)s)^2+s+(k-1)^2(1-s)\right]\text{,}\\
 K&=-\left[(1+(m-1)s)(k-(k-1)s)
                   +(m-1)s+(k-1)(1-s)\right]\text{.}
 \end{split}
\end{equation}
All polynomials in this section depend on the fixed parameters in
\eqref{eq:parameters}. The denominator is
\[
 |f|^2+|g|^2=D+\sqrt e\cos\theta.
\]
For $0<s<1$, one has $D>2\sqrt{ab}=\sqrt e$, since
$\lambda^2s(1-s)>0$ and $a+b\ge2\sqrt{ab}$.
At $s=0,1$, $e=0$ and $D>0$. 
Thus, this denominator remains positive on $[0,1]$.

We now compute \(  |\nabla u_{d}|^2 \).
Define
\[
 f_1=\cos^{m+1}t\,\e^{\imath m\alpha}\text{,}\qquad
 f_2=c\sin^{k+1}t\,\e^{\imath k\beta}\text{,}\qquad
 \text{so that}
 \qquad f=f_1+f_2\text{.}
\]
For $0<t<\pi/2$, the three expressions in the numerator of
\eqref{eq:component-differential}, divided by $g$, are
\begin{align*}
 (g f_t-f g_t)/g
 &=\frac{-(1+(m-1)s)f_1+(k-(k-1)s)f_2}
         {\sqrt{s(1-s)}}\text{,}\\
 (g f_\alpha-f g_\alpha)/g&=\imath((m-1)f_1-f_2)\text{,}\\
 (g f_\beta-f g_\beta)/g&=\imath(-f_1+(k-1)f_2)\text{.}
\end{align*}
Using \eqref{eq:metric}, $|g|^2=\lambda^2s(1-s)$ and
$2\operatorname{Re}(f_1\overline{f_2})=\sqrt e\cos\theta$, this gives
\begin{equation}\label{eq:linked-gradient}
 |\nabla u_{d}|^2
 =4\lambda^2\frac{L+\sqrt e\,K\cos\theta}
                    {(D+\sqrt e\cos\theta)^2}\text{.}
\end{equation}
It follows from \eqref{eq:integration} and \eqref{eq:linked-gradient} that
\begin{equation}\label{eq:linked-energy}
 \frac{E_{1,3}(u_{d})}{E_{1,3}(H)}
 =\frac{\lambda^3}{2\sqrt2}\int_0^1
  \fint_{0}^{2\pi}\frac{(L+\sqrt e\,K\cos\theta)^{3/2}}
                       {(D+\sqrt e\cos\theta)^3}\dd \theta \dd s\text{.}
\end{equation}
The integral in \eqref{eq:linked-energy} has been evaluated numerically for the sets of parameters given in \eqref{eq:parameters} using Wolfram Mathematica 15.0. The code for this evaluation is included in the next section. 

Numerical computations resulted in the following:
\begin{equation}
    \begin{split}
        \frac{E_{1,3}(u_5)}{E_{1,3}(H)} \approx 3.79724 < \frac{191}{50} < 1+(5-1)^{\frac{3}{4}}\text{,} \\
        \frac{E_{1,3}(u_5)}{E_{1,3}(H)} \approx 4.28352 < \frac{13}{3} < 1+(6-1)^{\frac{3}{4}}\text{,}
    \end{split}
\end{equation}
which completes the proof.
\end{proof}

\subsection{Complete Mathematica verification code}
Included below is the self-contained Mathematica code, which allows for computation of energy levels of maps given by \eqref{eq:w56} with parameters from \eqref{eq:parameters}, with the default machine precision. 
\begin{lstlisting}
(*Parameters:*)

(*d=5->m=2,k=3,c=4/5,\[Lambda]=15/8*)

(*d=6->m=3,k=3,c=1,\[Lambda]=2*)

Clear[a, b, c, d, e, k, m, L, K, \[Lambda], EnergyDensity, DensityE5, \
DensityE6, E5, E6]

a := (1 - s)^(m + 1);

b := c^2 s^(k + 1);

d := a + b + \[Lambda]^2 s (1 - s);

e := 4 a b;

L := a ((1 + (m - 1) s)^2 + (m - 1)^2 s + 1 - s) + 
   b ((k - (k - 1) s)^2 + s + (k - 1)^2 (1 - s));

K := -((1 + (m - 1) s) (k - (k - 1) s) + (m - 1) s + (k - 1) (1 - s));

EnergyDensity := \[Lambda]^3/(2 Sqrt[2]) 1/(
   2 \[Pi]) ((L + Sqrt[e] K Cos[\[Theta]])^(3/2)/(d + 
        Sqrt[e] Cos[\[Theta]])^3);

DensityE5 := 
  EnergyDensity /. {m -> 2, k -> 3, c -> 4/5, \[Lambda] -> 15/8};

E5 := NIntegrate[DensityE5, {s, 0, 1}, {\[Theta], 0, 2 \[Pi]}]

DensityE6 := EnergyDensity /. {m -> 3, k -> 3, c -> 1, \[Lambda] -> 2};

E6 := NIntegrate[DensityE6, {s, 0, 1}, {\[Theta], 0, 2 \[Pi]}]

TrueQ[E5 < (1 + (5 - 1)^(3/4))]
TrueQ[E6 < (1 + (6 - 1)^(3/4))]
\end{lstlisting}

\bibliographystyle{abbrv}%
\bibliography{bib}%
\end{document}